\documentclass[final]{amsart}

\usepackage{booktabs}
\usepackage{fullpage}
\usepackage[colorlinks,citecolor=cyan]{hyperref}
\usepackage{tristan}

\usepackage{pgfplots}
\usepackage{pgfplotstable}
\usepackage{subcaption}
\usepackage{algorithm}
\usepackage{algorithmicx}
\usepackage{algpseudocode}
\usepackage{tikz}
\usepackage{wrapfig}
\usepackage{capt-of}
\definecolor{fig}{RGB}{76,78,92}
\usepackage[labelfont=bf,font={color=fig,small}]{caption}
\usepackage[giveninits=true,eprint=false,url=false,style=alphabetic]{biblatex}
\definecolor{seaborngreen}{rgb}{0.3333333333333333, 0.6588235294117647, 0.40784313725490196}

\title{A Positivity-Preserving Finite Element Framework for Accurate Dose Computation in Proton Therapy}

\author[1]{Ben S. Ashby}

\author[1]{Abdalaziz Hamdan}

\author[1,2]{Tristan Pryer}

\address{$^1$ Institute for Mathematical Innovation\\ University of
  Bath, Bath, UK. $^2$ Department of Mathematical Sciences
  \\ University of Bath, Bath, UK.}

\begin{document}

\begin{abstract}
We present a stabilised finite element method for modelling proton
transport in tissue, incorporating both inelastic energy loss and
elastic angular scattering. A key innovation is a
positivity-preserving formulation that guarantees non-negative fluence
and dose, even on coarse meshes. This enables reliable computation of
clinically relevant quantities for treatment planning. We derive a
priori error estimates demonstrating optimal convergence rates and
validate the method through numerical benchmarks. The proposed
framework provides a robust, accurate and efficient tool for advancing
proton beam therapy.
\end{abstract}

\maketitle

\section{Introduction}

Proton Beam Therapy (PBT) has emerged as a promising modality for
treating cancers where conventional radiotherapy fails to sufficiently
spare surrounding healthy tissue. These include pediatric cases,
skull-base tumours and complex head and neck malignancies.

A proton deposits energy as it traverses matter, with energy
deposition increasing with depth and peaking at the end of its path, a
phenomenon known as the Bragg peak, illustrated in
Figure~\ref{fig:braggpeak}. This dose localisation underpins the
appeal of PBT, it enables precise tumour targeting while reducing
irradiation to critical structures.

\begin{wrapfigure}{r}{0.5\linewidth}
  \includegraphics[width=\linewidth]{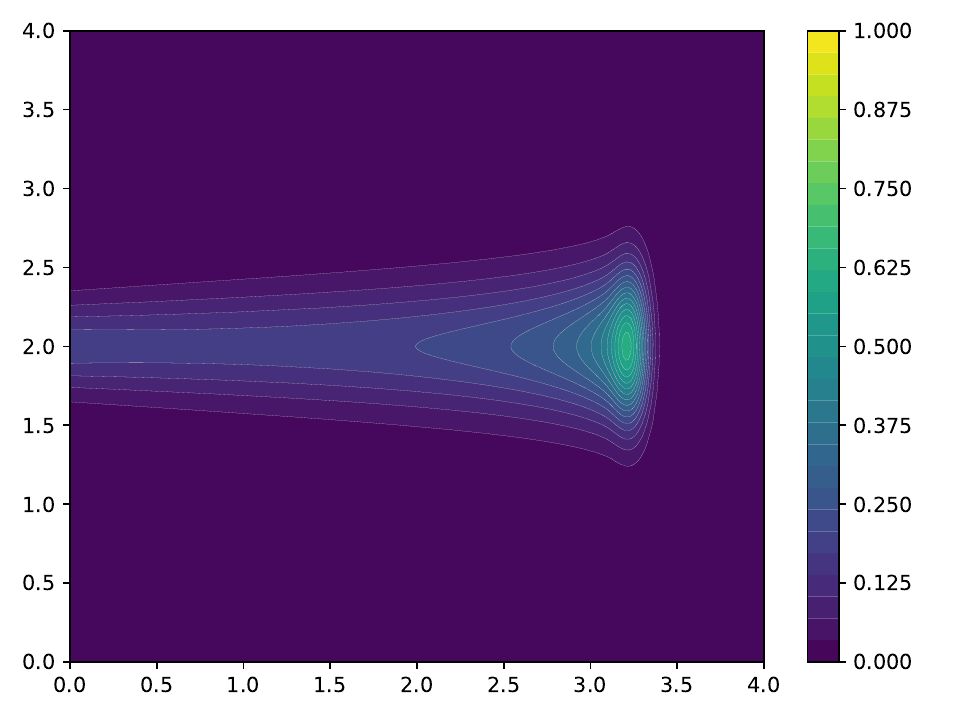}
  \captionof{figure}{
    \label{fig:braggpeak}
    A 2D slice of dose from a 62\,MeV proton beam. The sharp Bragg
    peak near the end of the track highlights the potential of proton
    beams to localise dose. This deterministic simulation includes
    inelastic and Coulomb scattering but neglects nuclear
    interactions.  }
\end{wrapfigure}

The potential for improved dose distributions has been recognised for
decades \cite{lomax_intensity_2003}, and by 2021, over 275{,}000
patients had received PBT worldwide
\cite{particle_therapy_co-operative_group_ptcog_particle_nodate}. Nonetheless,
conclusive evidence of its superiority over photon-based treatments
remains limited~\cite{chen_proton_2023}.

A persistent challenge in PBT is inter-fractional anatomical
variation, such as changes in patient hydration or positioning. These
perturbations introduce uncertainties in the delivered dose, often
mitigated by enlarging treatment margins, compromising the very
advantage of protons. Improving robustness thus hinges on predictive,
high-fidelity models of proton transport.

From a computational perspective, accurately modelling proton
transport and dose deposition is essential to addressing these
challenges. Protons interact with matter through inelastic collisions
(leading to continuous energy loss), elastic scattering (which
deflects particle direction), and non-elastic nuclear reactions
(producing secondaries); see Figure~\ref{fig:atom}. These processes
make high-fidelity simulation computationally intensive, creating a
bottleneck for routine clinical use.

In this work, we develop a deterministic finite element model for
proton transport, designed to ensure numerical accuracy and preserve
key physical properties. Deterministic solvers provide efficient
alternatives to Monte Carlo methods, but must contend with issues such
as nonphysical negativity and the stiffness of kinetic-type equations.

Our model incorporates both inelastic and elastic scattering into a
computationally efficient PDE-based framework. Continuous energy loss
is treated using the transport model of~\cite{ashby2025efficient},
while angular broadening is handled via a Laplace--Beltrami diffusion
operator on the sphere~\cite{dziuk1988finite,dedner2013analysis}. A
key innovation is a positivity-preserving discretisation that enforces
nonnegativity at Lagrange nodes using a variational inequality
formulation~\cite{amiri2024nodally,barrenechea2024nodally,ashby2025nodally}. This
addresses a long-standing issue in deterministic transport, where
standard methods produce unphysical negative fluences or
doses~\cite{stammer2024deterministic}.

We also propose three approaches for dose computation, Galerkin
projection, piecewise constant projection and a variational
inequality, which trade off between accuracy, simplicity and
positivity preservation. This makes our framework suitable for use in
downstream optimisation and inverse problems. Our techniques
generalise to other observables such as LET and biological response
models~\cite{ashby2025efficient}.

We prove optimal convergence rates for both fluence and dose in the
SUPG and positivity-preserving schemes. These theoretical results are
validated through numerical benchmarks, including comparisons with
Monte Carlo and semi-analytic solutions. The results demonstrate that
our method yields stable, accurate and physically meaningful
approximations, making it a practical tool for treatment planning and
future optimisation in proton therapy.

Monte Carlo methods have traditionally served as the gold standard for
modelling proton transport due to their ability to accurately capture
individual proton
interactions~\cite{salvat2013generic,jabbari2014fast}. However, their
high computational cost often limits their practicality in clinical
contexts, particularly for real-time treatment planning and dose
delivery verification \cite{CoxHattamKyprianouPryer:2023}. In
contrast, deterministic modelling of proton transport remains
relatively underexplored when compared to the extensive literature on
neutron transport and linear Boltzmann transport more broadly. Recent
efforts, such as~\cite{stammer2024deterministic}, have begun to
address this gap by developing deterministic solvers that balance
computational efficiency with physical fidelity. These approaches
borrow heavily from the numerical and analytical theory developed for
the linear Boltzmann equation, which models transport in space, angle,
energy and time.

For instance,~\cite{houston2024efficient} investigates fully
discontinuous Galerkin (dG) discretisations of linear Boltzmann
transport, treating the angular variable as an independent
coordinate. This leads to a high-dimensional problem: in spatial
dimension $d$, the PDE spans $d$ spatial, $d-1$ angular and one energy
variable. The resulting computational burden has motivated the
development of iterative and acceleration techniques, such as the
improved source iteration in~\cite{houston2024iterative},
quadrature-free methods in~\cite{radley2023quadrature} and efficient
sweep-based solvers in~\cite{calloo2024cycle}. These works aim to
reduce the cost of solving kinetic equations without sacrificing
accuracy or stability.

In deterministic settings, elastic scattering is often modelled using
simplified diffusion operators or approximate angular kernels. The
Laplace--Beltrami operator provides a geometrically natural model for
angular diffusion and has been applied within the context of
stochastic differential equation modelling for proton transport
\cite{crossley2025jump}. Inelastic scattering, which describes
continuous energy loss, is typically encoded via stopping power models
such as Bethe–Bloch. More recently,~\cite{ashby2025efficient}
introduced a PDE-based transport framework that captures continuous
slowing down efficiently, providing an alternative to stochastic
particle simulations.

Ensuring physical fidelity in deterministic models requires the
development of positivity-preserving schemes. Without such mechanisms,
standard discretisations may produce unphysical negative values for
fluence or dose, especially when working with coarse meshes. This
issue has been extensively studied in related elliptic and
convection-dominated problems. For
instance,~\cite{barrenechea2024nodally} proposes a variational
inequality framework that enforces positivity at nodal degrees of
freedom. The discrete solution lies in a convex set defined by bounds
on the nodal values, ensuring physical admissibility. These ideas
provide the foundation for our treatment of positivity in proton
transport.

The remainder of this paper is structured as follows. In
\S\ref{sec:model}, we introduce the mathematical model for charged
particle transport, including the representation of elastic and
inelastic scattering processes and the formulation of the governing
partial differential equation. \S\ref{sec:supg} develops the
streamline-upwind Petrov--Galerkin (SUPG) discretisation, including
the positivity-preserving variational inequality scheme and associated
error analysis. In \S\ref{sec:dose}, we consider numerical
approximation of the absorbed dose, propose several projection
strategies and derive bounds for the induced dose error. Numerical
results are presented in \S\ref{sec:numerics}, where we validate our
method against analytic benchmarks and high-fidelity Boltzmann
solutions and demonstrate the accuracy and stability of the proposed
formulations. Finally, \S\ref{sec:conclusion} summarises the main
findings and outlines directions for future research.

\section{Modelling of charged particle transport}
\label{sec:model}

In this section we introduce fundamental modelling concepts in proton
transport and discuss a simple model to aid in the exploration of
these ideas.

\begin{figure}[h!]
    \centering
    \renewcommand{\proton}[1]{%
    \shade[ball color=red!80!white, draw=black, line width=0.5pt] (#1) circle (.25);
    \draw[black] (#1) node{$+$};
}

\renewcommand{\neutron}[1]{%
    \shade[ball color=lime!70!green, draw=black, line width=0.5pt] (#1) circle (.25);
}

\renewcommand{\electron}[3]{%
    \draw[rotate = #3, color=gray!60!white, line width=0.8pt] (0,0) ellipse (#1 and #2);
    \shade[ball color=yellow!90!white, draw=black, line width=0.5pt] (0,#2)[rotate=#3] circle (.1);
}

\newcommand{\legendelectron}[1]{%
    \shade[ball color=yellow!90!white, draw=black, line width=0.5pt] (#1) circle (.1);
}

\renewcommand{\nucleus}{%
    \neutron{0.1,0.3}
    \proton{0,0}
    \neutron{0.3,0.2}
    \proton{-0.2,0.1}
    \neutron{-0.1,0.3}
    \proton{0.2,-0.15}
    \neutron{-0.05,-0.12}
    \proton{0.17,0.21}
}

\renewcommand{\inelastic}[2]{
  \proton{#1,#2};
  \draw[->,thick,cyan!80!white](#1+0.5,#2)--(4,-3.7); 
  \draw[->,thick,cyan!80!white](0,-3)--(-0.3,-4); 
  \shade[ball color=yellow] (-0.3,-4) circle (.1); 
}

\renewcommand{\elastic}[2]{
  \proton{#1,#2};
  \draw[->,thick,orange,bend right=90](#1+0.5,#2) to  [out=-30, in=-150] (4,3.);
}

\renewcommand{\protoncollision}[3]{
  \proton{#1,#2};
  \draw[->,thick,red](#1+0.5,#2)--(-0.5,0);%
  \draw[snake=coil, line after snake=0pt, segment aspect=0,%
    segment length=5pt,color=red!80!blue] (0,0)-- +(4,2)%
  node[fill=white!70!yellow,draw=red!50!white, below=.01cm,pos=1.]%
  {$\gamma$};%
  \draw[->,thick,red](#1+0.5,#2)--(-0.5,0);%
  \draw[->,thick,red](0.5,0)--(3.7,-1.8);%
  \neutron{4,-2};  
}

\begin{tikzpicture}[scale=0.5]
    \nucleus
    \electron{1.5}{0.75}{80}
    \electron{1.2}{1.4}{260}
    \electron{4}{2}{30}
    \electron{4}{3}{180}
    \protoncollision{-6.}{0.}{160}
    \inelastic{-6.}{-2.}
    \elastic{-6.}{2.}

         % Legend
      \begin{scope}[shift={(8,-1)}]        
        \draw [thick,rounded corners=2pt] (0,-4) rectangle (8,4); 
        \node at (4, 3.5) {\textbf{Legend}};        
        \proton{0.5, 3}
        \node[anchor=west, font=\footnotesize] at (1.5,3) {Proton}; 
        \neutron{0.5, 2}
        \node[anchor=west, font=\footnotesize] at (1.5,2) {Neutron};         
        \legendelectron{0.5, 1}
        \node[anchor=west, font=\footnotesize] at (1.5,1) {Electron};        
        \draw[->,thick,red] (0.5,0) -- +(1,0);
        \node[anchor=west, font=\footnotesize] at (1.5,0) {Nonelastic collision};        
        \draw[->,thick,cyan!80!white] (0.5,-1) -- +(1,0);
        \node[anchor=west, font=\footnotesize] at (1.5,-1) {Inelastic interaction};         
        \draw[->,thick,orange] (0.5,-2) -- +(1,0);
        \node[anchor=west, font=\footnotesize] at (1.5,-2) {Elastic interaction}; 
        \draw[snake=coil, line after snake=0pt, segment aspect=0,
          segment length=5pt,color=red!80!blue] (0.5,-3) -- +(1,0);
        \node[anchor=west, font=\footnotesize] at (1.5,-3) {prompt-$\gamma$ emission};
      \end{scope}
\end{tikzpicture}
    \caption{\em The three main interactions of a proton with matter.
      A \textcolor{red}{nonelastic} proton--nucleus collision, an
      \textcolor{cyan!80!white}{inelastic} Coulomb interaction with
      atomic electrons and \textcolor{orange}{elastic} Coulomb
      scattering with the nucleus.
      \label{fig:atom}
    }
\end{figure}
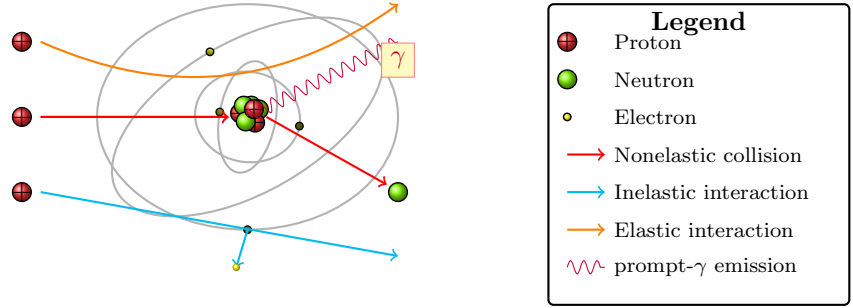
Consider a domain $\W \subset \reals^{d+1}$ for $d = 2, 3$, where $\W
= \W_{\vec{x}} \times \W_E$, with $\W_{\vec{x}} \subset \reals^d$
representing space and $\W_E = [E_{\min}, E_{\max}] \subset \reals$
energy.

\subsection{Inelastic scattering}

For $\alpha > 0$ and $p \in [1,2]$ we introduce
\begin{equation}\label{eq:BraggKleeman}
  S(E) = \frac{1}{\alpha p} E^{1-p},
\end{equation}
the Bragg--Kleeman formula \cite{BraggKleeman:1905}, which illustrates
how stopping power decreases with increasing energy. That is, $S$ is
monotonically decreasing. From the PDE perspective, this provides a
dissipative mechanism, motivating the definition
\begin{equation}
  \mu = -S'(E_{\min}) > 0.
\end{equation}

Empirical values for $p$ and $\alpha$ are provided in Table
\ref{tab:range_energy}, showing the variation of $\alpha$ across
different biological media, while $p$ remains relatively constant.
\begin{table}[h!]
  \centering
  \begin{tabular}{lrr}
    \toprule
    Medium &    $p$ &  $\alpha$ \\
    \midrule
    Water & 1.75$\pm$0.02 & 0.00246$\pm$0.00025 \\
    Muscle & 1.75 &            0.0021 \\
    Bone & 1.77 &            0.0011 \\
    Lung & 1.74 &            0.0033 \\
    \bottomrule
  \end{tabular}
  \caption{
    \label{tab:range_energy}
    Range--energy relationship parameters for different media. The
    parameter $p$ remains relatively constant across different
    biological media, while $\alpha$ varies significantly with
    density and composition.}
\end{table}

For fixed unitary angle $\vec \omega \in \mathbb{S}^{d-1}$, we treat
direction as a parameter rather than an independent variable, thereby
reducing the angular domain to a point. In the full Boltzmann
framework, $\vec \omega$ ranges over the unit sphere and the
Laplace--Beltrami operator acts on the angular variable. In this
setting, however, we consider $\vec \omega$ fixed and examine the
implications for modelling angular diffusion.

The operator $\Delta_{\vec \omega} u$ denotes the Laplace--Beltrami
operator on the sphere, which formally takes the form
\begin{equation}
  \Delta_{\vec \omega} u = \nabla_{\vec \omega} \cdot (\nabla_{\vec \omega} u),
\end{equation}
where the angular gradient is defined using the projection operator
\begin{equation}
  \mathcal{P} = \mathbf{I} - \vec{\omega} \otimes \vec{\omega},
\end{equation}
which projects vectors onto the tangent space of $\mathbb{S}^{d-1}$ at
$\vec{\omega}$. Specifically, for a function $u$, the projected
gradient is
\begin{equation}
  \nabla_{\vec \omega} u
  =
  \mathcal{P} \nabla_{\vec{x}} u
  =
  (\mathbf{I} - \vec{\omega} \otimes \vec{\omega}) \nabla_{\vec{x}} u
  =
  \nabla_{\vec{x}} u - \qp{\vec{\omega} \cdot \nabla_{\vec{x}} u} \vec \omega.
\end{equation}

\begin{remark}[2D projection identity]
  For $d = 2$, the projection of the gradient orthogonal to a unit
  vector $\vec{\omega}$ can be written compactly using the
  perpendicular vector $\vec{\omega}^\perp$, defined as  
  \begin{equation}
    \vec{\omega}^\perp = 
    \begin{pmatrix}
      -\omega_2 \\
      \omega_1
    \end{pmatrix}
    \quad \text{for} \quad
    \vec{\omega} = 
    \begin{pmatrix}
      \omega_1 \\
      \omega_2
    \end{pmatrix}.
  \end{equation}
  Then
  \begin{equation}
    (\mathbf{I} - \vec{\omega} \otimes \vec{\omega}) \nabla u
    =
    (\vec{\omega}^\perp \cdot \nabla u)\, \vec{\omega}^\perp,
  \end{equation}
  expressing the orthogonal projection of $\nabla u$ as a scalar
  multiple of a single direction.

  For $d = 3$, the orthogonal complement of $\vec \omega$ is
  two-dimensional, so no unique analogue of $\vec{\omega}^\perp$
  exists. The projection must instead be expressed in terms of a full
  basis for the plane orthogonal to $\vec \omega$.
\end{remark}

This representation assumes $u$ is smooth enough for all
quantities to be well defined. Fixing $\vec \omega$ eliminates the
angular domain and replaces the Laplace--Beltrami operator with an
artificial angular diffusion term involving an implicit boundary
condition. Since $\mathbb{S}^{d-1}$ is a compact manifold without
boundary, any notion of inflow or Dirichlet data on the angular domain
is induced entirely by this parametrisation.

Given $\vec \omega \in \mathbb{S}^{d-1}$, $f \in \leb{2}(\W)$ and $g
\in \leb{2}(\partial \W)$, we introduce the degenerate elliptic proton
transport problem: find the fluence $\psi \colon \W \to \reals$ such
that
\begin{equation}\label{eq:pde_classical_form}
  \begin{split}
    \vec{\omega} \cdot \nabla_{\vec{x}} \psi(\vec{x}, E)
    -
    \dfrac{\partial}{\partial E} \left( \mathscr S(E) \psi(\vec{x}, E) \right)
    - \epsilon \Delta_{\vec \omega} \psi(\vec x, E)
    &= 
    f \quad \text{in } \W, \\
    \psi &= g \quad \text{on } \Gamma_-.
  \end{split}
\end{equation}
The inflow boundary $\Gamma_- \subset \partial \W$ is given by
\begin{equation*}
  \Gamma_- 
  =
  \{\vec{x} \in \partial \W_{\vec{x}} : \vec{\omega} \cdot \vec{n}_{\vec x} < 0\}
  \cup
  \{\partial\W_E = E_{\max}\},
\end{equation*}
where $\vec{n}_{\vec x}$ denotes the outward unit normal vector to
$\partial \W_{\vec{x}}$. The problem domain is illustrated in
Figure~\ref{fig:domain}.

\begin{figure}[h!]
  \begin{center}
    \begin{tikzpicture}
    % Define dimensions
    \def\zmax{5}
    \def\Emin{0.5}
    \def\Emax{2.5}
    \def\xmax{4}

    \def\shiftX{2.5}
    \def\shiftY{1.5}

    \draw[thick] (\shiftX, \Emin+\shiftY) -- (\shiftX+\zmax, \Emin+\shiftY) -- (\shiftX+\zmax, \Emax+\shiftY) -- (\shiftX, \Emax+\shiftY) -- cycle;

    \draw[thick] (0, \Emin) -- (\zmax, \Emin) -- (\zmax, \Emax) -- (0, \Emax) -- cycle;

    \draw[thick] (0, \Emin) -- (\shiftX, \Emin+\shiftY);
    \draw[thick] (\zmax, \Emin) -- (\zmax+\shiftX, \Emin+\shiftY);
    \draw[thick] (\zmax, \Emax) -- (\zmax+\shiftX, \Emax+\shiftY);
    \draw[thick] (0, \Emax) -- (\shiftX, \Emax+\shiftY);
    \node[below] at (0, \Emin) {\footnotesize 0};
    \node[below] at (\zmax, \Emin) {\footnotesize $z_{\text{max}}$};
    \node[left] at (0, \Emin) {\footnotesize $E_{\text{min}}$};
    \node[left] at (0, \Emax) {\footnotesize $E_{\text{max}}$};
    \node[below] at (\shiftX+0.5, \Emin+\shiftY) {\footnotesize $x_{\text{max}}$};

    % Draw inflow boundary (left face, top face, and front face)
    \fill[orange, opacity=0.2] (0, \Emin) -- (0, \Emax) -- (\shiftX, \Emax+\shiftY) -- (\shiftX, \Emin+\shiftY) -- cycle;
    \fill[orange, opacity=0.2] (0, \Emax) -- (\zmax, \Emax) -- (\zmax+\shiftX, \Emax+\shiftY) -- (\shiftX, \Emax+\shiftY) -- cycle;  
    % Label inflow boundary
    \node[orange] at (-0.3, 1.5) {\footnotesize $\partial X_-$}; 
    \node[orange] at (\zmax/2+1, \Emax+0.5) {\footnotesize $\partial X_-$};

\end{tikzpicture}    
    \caption{\label{fig:domain} Illustration of the domain and the
      relevant inflow boundary for $d = 2$ with $\vec \omega = (0,1)$, say.}
  \end{center}
\end{figure}
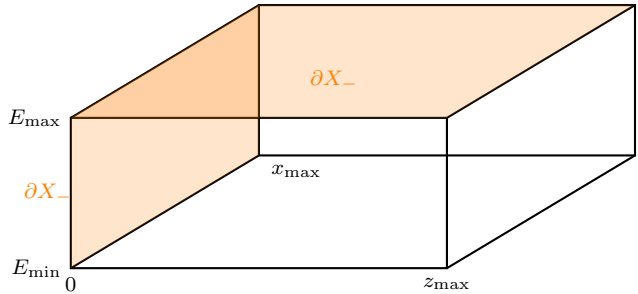

The analysis of this degenerate elliptic problem is challenging due to
the partial smoothing effects of the angular diffusion term, combined
with the advective behaviour. For example, while angular diffusion
regularises the solution in directions orthogonal to $\vec \omega$, the advection
term can propagate discontinuities along streamlines.

\begin{remark}[Physical interpretation]
  The proton transport equation models the evolution of fluence in
  both space and energy, incorporating the two dominant physical
  effects from Figure~\ref{fig:atom}:
  \begin{equation}\label{eq:pde_classical_annotated}
    \underbrace{
      \vec{\omega} \cdot \nabla_{\vec{x}} \psi
    }_{\text{Transport}}
    -
    \underbrace{
      \nabla_E \left( \mathscr S(E) \psi \right)
    }_{\text{\textcolor{cyan!80!white}{Inelastic Coulomb scattering}}}
    -
    \underbrace{
      \epsilon \Delta_{\vec \omega} \psi
    }_{\text{\textcolor{orange}{Elastic Coulomb scattering}}}
    =
    f \quad \text{in } \W.
  \end{equation}
  The first term represents directional transport, describing how
  protons travel in straight lines through the medium along a fixed
  direction.

  The second term accounts for energy loss due to inelastic Coulomb
  interactions, where protons transfer energy to bound electrons in
  the material. We model this dissipation using the Bragg--Kleeman
  stopping power formulation, a simple representation of the
  continuous slowing down of protons. Other models, such as
  Bethe--Bloch, are compatible with the framework but are not
  considered further here.

  The third term represents elastic Coulomb scattering, corresponding
  to angular deflections caused by interactions with atomic nuclei. This
  produces a diffusion-like effect in the angular variable, capturing
  the gradual broadening of the beam from small-angle scattering.

  The source term allows for an external input of protons within the
  domain, while the inflow boundary condition accounts for incident
  protons entering the system, either from a physical boundary or a
  prescribed beam source.
\end{remark}

\subsection{Weak formulation}

We work in the space $\sobh1(\W)$ and impose the inflow boundary
condition weakly through the bilinear and linear forms.  This space
allows a variational formulation of problem
\eqref{eq:pde_classical_form}: Find $\psi \in \sobh1(\W)$ such that
\begin{equation}\label{eq:pde_variational_form}
  \mathcal A(\psi, v)
  +
  \mathcal B(\psi, v)
  =
  l(v) \Foreach v \in \sobh1(\W),
\end{equation}
where 
\begin{equation}\label{eq:a_bilinear}
  \mathcal A(u, v)
  =
  \int_{\W}
  \epsilon
  \nabla_{\vec \omega} u \cdot \nabla_{\vec \omega} v
  \d\vec{x} \, \d E,
\end{equation}
\begin{equation}\label{eq:b_bilinear}
  \mathcal B(u, v)
  =
  \int_{\W}
  \qp{\vec{\omega} \cdot \nabla_{\vec{x}}  u 
    -
    \nabla_E \qp{\mathscr{S}(E) u}
  }v
  \d\vec{x} \, \d E
  -
  \frac 12\int_{\Gamma_-} \qp{\vec \omega \cdot \vec n_{\vec x} - \mathscr S(E) n_E} u v \d s,
\end{equation}
and 
\begin{equation}\label{eq:l_linear}
  l(v)
  =
  \int_{\W}fv   \d\vec{x} \, \d E
  -
  \frac 12\int_{\Gamma_-} \qp{\vec \omega \cdot \vec n_{\vec x} - \mathscr S(E) n_E} g v \d s.
\end{equation}

\begin{remark}[Boundary conditions and projected diffusion]
  \label{rem:boundary}
  The inflow boundary condition is imposed weakly through the bilinear
  form $\mathcal{B}$ and linear form $l$. This inflow condition allows
  us to account for prescribed incident particles entering the domain,
  typically from a beam source.
  
  The diffusion operator acts only in directions orthogonal to $\vec
  \omega$. Since no boundary condition is imposed explicitly in those
  directions, we must assume a homogeneous Neumann condition
  \begin{equation}
    \nabla_{\vec \omega} \psi \cdot \vec n_{\vec x}
    =
    0 \quad \text{on } \Gamma^\perp(\vec \omega) \subset \partial \W_{\vec x},
  \end{equation}
  where
  \begin{equation}
    \Gamma^\perp(\vec \omega)
    =
    \ensemble{
      \vec x \in \partial \W_{\vec x}
    }{
      \mathcal{P} \vec n_{\vec x} \ne \vec 0
    }
    \quad \text{with } \mathcal{P} = \mathbf{I} - \vec \omega \otimes \vec \omega.
  \end{equation}
  This identifies the portion of the boundary where the outward normal
  has a nonzero component orthogonal to $\vec \omega$ and thus where
  transverse diffusion may induce a nontrivial normal flux.

  Physically, this corresponds to assuming that elastic Coulomb
  scattering, which induces transverse broadening of the beam, does
  not result in any loss of particles through the boundary in the
  directions orthogonal to transport. That is, particles scatter
  within the medium but do not escape through the sides. This is
  consistent with the interpretation of the domain as a confined
  spatial region in which the beam evolves.

  For $d=2$, this condition reduces to checking whether $\vec n_{\vec
    x} \cdot \vec \omega^\perp \ne 0$. For example, on a square
  spatial domain if $\vec \omega = (1,0)$, then $\vec \omega^\perp =
  (0,1)$, and the Neumann condition becomes $\partial_{x_2} \psi = 0$
  on the top and bottom edges of the domain, where the outer normal
  has nonzero vertical component.
\end{remark}

\begin{remark}[Regularising effects and hypocoercivity]
  The angular diffusion term contributes to improved regularity of
  solutions by penalising high-frequency oscillations orthogonal to
  the transport direction. When combined with the advection term, the
  system exhibits a property known as \emph{hypocoercivity}, a
  mechanism by which the interplay between transport and dissipation
  leads to decay to equilibrium and enhanced regularity under suitable
  conditions \cite{villani2009hypocoercivity}. This effect is
  particularly relevant in degenerate elliptic problems, where
  coercivity may be lacking in some directions, but regularity is
  recovered through the coupling of transport and diffusion
  \cite{georgoulis2021hypocoercivity,porretta2017numerical}. See also
  \cite{pim2024optimal} for a discussion of hypocoercivity in the
  context of charged particle transport.
\end{remark}

\begin{remark}[Problems satisfying a priori bounds]
  For some choices of $\mathscr S$ and $\epsilon = 0$, problem
  \eqref{eq:pde_classical_form} can be solved analytically using the
  method of characteristics \cite{ashby2025efficient}. When $\epsilon
  > 0$, the angular diffusion term provides a regularising effect that
  suppresses transverse oscillations, but the problem is no longer
  analytically tractable in general. In either case, the structure of
  the equation allows for the derivation of useful \emph{a priori}
  bounds on $\psi$, depending on the data $f$ and $g$.
\end{remark}

\begin{remark}[Maximum principle]
  The solution to problem \eqref{eq:pde_classical_form} satisfies a
  maximum principle. Specifically, if $f = 0$ and $g \geq 0$, then the
  solution is nonnegative and bounded above by the incoming boundary
  data
  \begin{equation*}
    0
    \leq \psi(\vec{x}, E)
    \leq \sup_{\Gamma_-} g, \quad \forall (\vec{x}, E) \in \W.
  \end{equation*}
  This follows from the directional structure of the equation: the
  advection term transports information along characteristics from
  $\Gamma_-$ and the diffusion operator $\epsilon \Delta_{\vec
    \omega}$ acts only in directions orthogonal to transport, without
  introducing any additional source of negativity. In this setting,
  $\vec \omega$ is treated as a fixed parameter and no boundary
  condition is imposed in the directions where diffusion acts.

  This property is important to ensure nonnegativity of the fluence
  and is essential for the physical interpretability of the model;
  see, e.g., \cite{renardy2006introduction}. It also forms the core
  principle motivating the design of our numerical scheme in the
  following section.
\end{remark}

\section{SUPG discretisation}
\label{sec:supg}

We now introduce a stabilised finite element method for the
variational problem \eqref{eq:pde_variational_form}. The key idea is
to retain the diffusion term in its standard Galerkin form, while
modifying the treatment of the advective terms using a Streamline
Upwind Petrov--Galerkin (SUPG) approach. This stabilisation introduces
a consistent perturbation to the test space along the characteristic
direction of transport, enhancing stability and suppressing
nonphysical oscillations in the numerical solution. The method
preserves the variational structure of the problem and is compatible
with the weak enforcement of boundary conditions described previously.

Let $\mathcal{T}$ be a regular subdivision of $\W$ into disjoint
simplicial or box-type (quadrilateral/hexahedral) elements $K$, with
$\bar{\W} = \bigcup_{K \in \mathcal{T}} \bar{K}$. We assume
$\mathcal{T}$ is shape-regular and that the elemental faces are
straight planar segments. Let $h_K = \text{diam}(K)$, and define the
piecewise constant meshsize function $h$ by $h|_K = h_K$.

For $k \in \mathbb{N}$, we define the continuous finite element space
\begin{equation}
  \fes
  :=
  \ensemble{v_h \in \leb2(\W)}{v_h\vert_{K} \in \mathbb{P}^k(K) \Foreach K \in \mathcal{T}} \cap \cont{0}(\W),
\end{equation}
and denote the set of Lagrange nodes by $\{\vec x_i\}_{i=1}^{\dim\fes}$.

\subsection{Stabilised formulation}

For $u_h, v_h \in \fes$, define the transport operator
\begin{equation}
  \cL(u_h)
  :=
  \vec{\omega} \cdot \nabla_{\vec{x}} u_h - \nabla_E \qp{\mathscr{S}(E) u_h},
\end{equation}
and consider the stabilised bilinear form
\begin{equation}
  \begin{split}
    \mathcal{B}_h(u_h, v_h)
    &:=
    \int_\W \cL(u_h) v_h \d \vec x \d E
    -
    \frac 12\int_{\Gamma_-} \qp{\vec \omega \cdot \vec n_{\vec x} - \mathscr S(E) n_E} u_h v_h \d s
    \\
    &\quad +
    \sum_{K \in \mathcal{T}}
    \int_K
    \delta_K \cL(u_h) \cL(v_h)
    \d \vec x \d E
    \\
    &=
    \int_\W
    \qp{ \vec{\omega} \cdot \nabla_{\vec{x}} u_h
      -
      \nabla_E(\mathscr{S}(E) u_h) } 
    v_h
    \d\vec{x} \d E
    \\
    &\quad -
    \frac 12 \int_{\Gamma_-} \qp{\vec \omega \cdot \vec n_{\vec x} - \mathscr S(E) n_E} u_h v_h \d s
    \\
    &\quad +
    \sum_{K \in \mathcal{T}}
    \int_K
    \delta_K
    \qp{\vec{\omega} \cdot \nabla_{\vec{x}} u_h - \nabla_E (\mathscr{S}(E) u_h)}
    \qp{\vec{\omega} \cdot \nabla_{\vec{x}} v_h - \nabla_E (\mathscr{S}(E) v_h)}
    \d\vec{x} \d E.
  \end{split}
\end{equation}
Here, $\delta_K$ is a stabilisation parameter chosen based on the
local mesh size and characteristic transport speed to be specified
below.

The stabilised finite element method then reads: find $\psi_h \in
\fes$ such that
\begin{equation}
  \label{eq:SUPGsoln}
  \mathcal A(\psi_h, v_h)
  +
  \mathcal B_h(\psi_h, v_h)
  =
  l(v_h)
  \Foreach v_h \in \fes.
\end{equation}

\begin{remark}[Positivity in deterministic radiation transport simulations]
  Even in simplified settings, standard finite element solvers may
  fail to produce physically meaningful, nonnegative solutions
  \cite{ashby2025nodally}. This issue becomes more pronounced in
  realistic configurations \cite{stammer2024deterministic}; see also
  the example in \S\ref{sec:numerics}. In the next section, we
  describe a modification of the SUPG scheme \eqref{eq:SUPGsoln} that
  enforces positivity at the degrees of freedom, using techniques
  developed in \cite{ashby2025nodally}.
\end{remark}

\subsection{A positivity-preserving scheme}

Let $\{\vec x_i\}_{i=1}^N$ denote the Lagrange nodes of $\fes$. We
define the convex subset $K_h \subset \fes$ consisting of all
functions whose nodal values lie within the physical bounds:
\begin{equation}
  K_h 
  := 
  \ensemble{v_h \in \fes}{v_h(\vec x_i) \in [0, \sup_{\Gamma_-} g] \Foreach i = 1, \dots, N}.
\end{equation}

\begin{remark}[Preservation of bounds at the degrees of freedom]
  For polynomial degree 1, functions in $K_h$ are guaranteed to
  satisfy the upper and lower bounds pointwise, as they are linearly
  interpolated between nodal values. For polynomial degree 2 or
  higher, this is no longer the case: functions in $K_h$ are only
  \emph{nodally} bound-preserving. This notion was studied in
  \cite{barrenechea2024nodally} for second-order elliptic problems,
  where nodal bound preservation was achieved via a nonlinear
  stabilised method. There, the bound-preserving discrete solution was
  shown to coincide with the solution of a discrete variational
  inequality.
\end{remark}

We now seek a solution $\psi_h^+ \in K_h$ satisfying the inequality
\begin{equation}
  \label{eq:discrete_inequality}
  \mathcal A(\psi_h^+, v_h - \psi_h^+)
  +
  \mathcal B_h(\psi_h^+, v_h - \psi_h^+)
  \geq
  l(v_h - \psi_h^+) \quad \Foreach v_h \in K_h.
\end{equation}
By construction, $\psi_h^+ \in K_h$ satisfies a nodal analogue of the
maximum principle, preserving nonnegativity and an upper bound
determined by the inflow data.

We proceed to analyse the method from an \emph{a priori} perspective.

\begin{lemma}[Consistency]
  \label{lem:consistency}
  Let $\psi \in \sobh2(\W)$ solve \eqref{eq:pde_variational_form} and
  satisfy $\psi = g$ pointwise on $\Gamma_-$. Then
  \begin{equation}
    \mathcal A(\psi, v_h)
    +
    \mathcal B_h(\psi, v_h)
    =
    l(v_h)
    \Foreach v_h \in \fes.
  \end{equation}
\end{lemma}

\begin{proof}
  Since $\psi \in \sobh2(\W)$ and $\nabla_{\vec \omega} = \mathcal{P} \nabla_{\vec x}$, we may integrate by parts in the angular diffusion term:
  \begin{equation}
    \begin{split}
      \mathcal{A}(\psi, v_h)
      &=
      \epsilon
      \int_\W
      \nabla_{\vec \omega} \psi \cdot 
      \nabla_{\vec \omega} v_h
      \d\vec{x} \, \d E
      \\
      &=
      -\epsilon
      \int_\W
      \Delta_{\vec \omega} \psi \, v_h \,
      \d\vec{x} \, \d E
      +
      \epsilon
      \int_{\Gamma^\perp(\vec \omega)}
      (\nabla_{\vec \omega} \psi \cdot \vec{n}_{\vec x}) \, v_h
      \d s.
    \end{split}
  \end{equation}
  The integration by parts identity follows from standard divergence
  theory applied to projected gradients; see, e.g.,
  \cite[§1.9]{ern2004theory}.

  The boundary term vanishes due to the homogeneous Neumann condition
  \begin{equation}
    \nabla_{\vec \omega} \psi \cdot \vec n_{\vec x} = 0
    \quad \text{on } \Gamma^\perp(\vec \omega),
  \end{equation} 
  as discussed in Remark \ref{rem:boundary} above. Thus,
  \begin{equation}
    \mathcal{A}(\psi, v_h)
    =
    -\epsilon \int_\W \Delta_{\vec \omega} \psi \, v_h \d \vec x \d E.
  \end{equation}

  For the transport terms, since $\psi$ solves the PDE in the strong
  form, the residual of $\cL(\psi)$ vanishes pointwise. Therefore, the
  SUPG stabilisation term gives no contribution,
  \begin{equation}
    \sum_{K \in \mathcal{T}} \int_K \delta_K \cL(\psi) \cL(v_h) = 0.
  \end{equation}
 For the transport terms, since $\psi$ solves the PDE in the strong
  form, the residual of $\cL(\psi)$ vanishes pointwise. Therefore, the
  SUPG stabilisation term gives no contribution
  \begin{equation}
    \mathcal B_h(\psi, v_h)
    =
    \int_\W \cL(\psi) v_h \d \vec x \d E
    +
    \sum_{K \in \mathcal{T}} \int_K \delta_K \cL(\psi) \cL(v_h) \d \vec x \d E
    =
    \int_\W \cL(\psi) v_h \d \vec x \d E.
  \end{equation}
  Finally, since $\psi$ satisfies the variational formulation
  \eqref{eq:pde_variational_form}, we conclude that
  \begin{equation}
    \mathcal A(\psi, v_h) + \mathcal B_h(\psi, v_h) = l(v_h) \quad \Foreach v_h \in \fes.
  \end{equation}
\end{proof}

A natural notion of error when considering the problem
\eqref{eq:SUPGsoln} is the energy norm, defined by
\begin{equation}
  \begin{split}
    \enorm{u_h}^2
    &:=
    \epsilon \Norm{\nabla_{\vec{\omega}} u_h}_{\leb2(\W)}^2
    +
    \mu \Norm{u_h}_{\leb2(\W)}^2
    +
    \sum_{K \in \mathcal{T}}
    \Norm{\delta_K^{1/2} \cL(u_h)}_{\leb{2}(K)}^2
    +
    \frac 12 \int_{\Gamma^+}
    \qp{ \vec \omega \cdot \vec n_{\vec x} - \mathscr S(E) n_E } u_h^2 \d s.
  \end{split}
\end{equation}
Note that this is indeed a norm since $\vec \omega \cdot \vec n_{\vec
  x} - \mathscr S(E) n_E \geq 0$ on $\Gamma^+$.

\begin{lemma}[Continuity and coercivity]
  \label{lem:dg-coerc}
  Let $u_h\in \fes$. Then
  \begin{equation}
    \label{eq:dg-coerc-one}
    \mathcal A(u_h, u_h)
    +
    \mathcal B_h(u_h, u_h)
    \geq
    \enorm{u_h}^2.
  \end{equation}
  Furthermore, let
  \begin{equation}
    \enorm{u_h}_*^2
    :=
    \enorm{u_h}^2
    +
    \sum_{K\in\mathcal T}
    \delta_K^{-1} \Norm{u_h}_{\leb{2}(K)}^2
  \end{equation}
  Then, for $u_h,v_h\in \fes$
  \begin{equation}
    \mathcal A(u_h, v_h)
    +
    \mathcal B_h(u_h, v_h)
    \leq
    C_B
    \enorm{u_h}_*
    \enorm{v_h}.
  \end{equation}
\end{lemma}

\begin{proof}
  By definition we have
  \begin{equation}
    \begin{split}
      \mathcal{B}_h(u_h, u_h)
      &=
      \int_\W \cL(u_h) u_h \d \vec x \d E
    -
    \int_{\Gamma_-} \qp{ \vec \omega \cdot \vec n_{\vec x} - \mathscr S(E) n_E } u_h^2 \d s
    +
    \sum_{K \in \mathcal{T}} \Norm{ \delta_K^{1/2} \cL(u_h) }_{\leb2(K)}^2.
    \end{split}
  \end{equation}
  We expand the first term using the product rule
  \begin{equation}
    \nabla_E(\mathscr S(E) u_h)
    =
    \mathscr S'(E) u_h + \mathscr S(E) \nabla_E u_h.
  \end{equation}
  Then,
  \begin{equation}
    \begin{split}
      \int_\W \cL(u_h) u_h \d \vec x \d E
      &=
      \int_\W \vec \omega \cdot \nabla_{\vec x} u_h \, u_h \d \vec x \d E
      -
      \int_\W \qp{ \mathscr S'(E) u_h + \mathscr S(E) \nabla_E u_h } u_h \d \vec x \d E
      \\
      &=
      \frac{1}{2} \int_{\Gamma} \vec \omega \cdot \vec n_{\vec x} \, u_h^2 \d s
    -
    \frac{1}{2} \int_{\Gamma} \mathscr S(E) n_E \, u_h^2 \d s
    -
    \int_\W \mathscr S'(E) u_h^2 \d \vec x \d E.
  \end{split}
\end{equation}
Therefore,
\begin{equation}
  \begin{split}
    \mathcal{B}_h(u_h, u_h)
    &=
    \sum_{K \in \mathcal{T}} \Norm{ \delta_K^{1/2} \cL(u_h) }_{\leb2(K)}^2
    +
    \frac 12 \int_{\Gamma^+} \qp{ \vec \omega \cdot \vec n_{\vec x} - \mathscr S(E) n_E } u_h^2 \d s
    -
    \int_\W \mathscr S'(E) u_h^2 \d \vec x \d E.
  \end{split}
\end{equation}

The term involving $\mathscr S'(E)$ is strictly negative and we recall
that $\mu := -\mathscr S'(E_{\min}) > 0$. Therefore, using the
monotonicity of $\mathscr S$, we have
\begin{equation}
  \int_\W \mathscr S'(E) u_h^2 \d \vec x \d E \leq -\mu \Norm{u_h}_{\leb2(\W)}^2.
\end{equation}
The angular diffusion term contributes
\begin{equation}
  \mathcal{A}(u_h, u_h) = \epsilon \Norm{ \nabla_{\vec \omega} u_h }_{\leb2(\W)}^2.
\end{equation}
Combining all terms, we find
\begin{equation}
  \mathcal{A}(u_h, u_h) + \mathcal{B}_h(u_h, u_h)
  \geq
  \epsilon \Norm{ \nabla_{\vec \omega} u_h }_{\leb2(\W)}^2
  +
  \mu \Norm{u_h}_{\leb2(\W)}^2
  +
  \sum_K \Norm{ \delta_K^{1/2} \cL(u_h) }_{\leb2(K)}^2
  +
  \int_{\Gamma^+} \qp{ \vec \omega \cdot \vec n_{\vec x} - \mathscr S(E) n_E } u_h^2 \d s,
\end{equation}
showing coercivity.
  
For continuity, each term is estimated separately. For the angular
diffusion term,
\begin{equation}
  \mathcal{A}(u_h, v_h)
  =
  \epsilon \int_\W \nabla_{\vec \omega} u_h \cdot \nabla_{\vec \omega} v_h \d \vec x \d E
  \leq
  \epsilon
  \Norm{ \nabla_{\vec \omega} u_h }_{\leb2(\W)}
  \Norm{ \nabla_{\vec \omega} v_h }_{\leb2(\W)}.
\end{equation}
For the transport term, we use Cauchy--Schwarz on each element 
\begin{equation}
  \begin{split}
    \int_\W \cL(u_h) v_h
    =
    \sum_{K \in \mathcal{T}} \int_K \cL(u_h) v_h
    &\leq
    \sum_K
    \Norm{ \delta_K^{1/2} \cL(u_h) }_{\leb2(K)}
    \delta_K^{-1} \Norm{ v_h }_{\leb2(K)}.
    \\
    &\leq
    \qp{
      \sum_K \Norm{ \delta_K^{1/2} \cL(u_h) }_{\leb2(K)}^2
    }^{1/2}
    \qp{
      \sum_K \delta_K^{-1} \Norm{ v_h }_{\leb2(K)}^2
    }^{1/2}.
  \end{split}
\end{equation}
The stabilisation term satisfies
\begin{equation}
  \sum_K \int_K \delta_K \cL(u_h) \cL(v_h)
  \leq
    \qp{
      \sum_K \Norm{ \delta_K^{1/2} \cL(u_h) }_{\leb2(K)}^2
    }^{1/2}
    \qp{
      \sum_K \delta_K^{-1} \Norm{ v_h }_{\leb2(K)}^2
    }^{1/2}.
\end{equation}
For the inflow boundary term, observe that the weight $\vec \omega
\cdot \vec n_{\vec x} - \mathscr S(E) n_E$ is non-negative and
uniformly bounded above on $\Gamma_-$ by
\begin{equation}
  \left| \vec \omega \cdot \vec n_{\vec x} - \mathscr S(E) n_E \right|
  \leq 1 + \Norm{ \mathscr S(E) }_{\leb\infty(\W)}.
\end{equation}
Then, by a trace inverse inequality on shape-regular meshes (e.g.,
\cite{ern2004theory}), we obtain
\begin{equation}
  \int_{\Gamma_-}
  \qp{ \vec \omega \cdot \vec n_{\vec x} - \mathscr S(E) n_E } u_h v_h \, \mathrm{d}s
  \leq
  C_{\operatorname{tr}} \qp{1 + \Norm{ \mathscr S(E) }_{\leb\infty(\W)}}
  \sum_{K \in \mathcal{T}} \delta_K^{-1/2} \Norm{ u_h }_{\leb2(K)} \delta_K^{1/2} \Norm{ v_h }_{\leb2(K)}.
\end{equation}
Applying Cauchy--Schwarz over $K \in \mathcal{T}$, this yields
\begin{equation}
  \int_{\Gamma_-}
  \qp{ \vec \omega \cdot \vec n_{\vec x} - \mathscr S(E) n_E } u_h v_h \, \mathrm{d}s
  \leq
  C_{\operatorname{tr}} \qp{1 + \Norm{ \mathscr S(E) }_{\leb\infty(\W)}}
  \enorm{u_h}_* \enorm{v_h},
\end{equation}
which, upon combining with the previous estimates, completes the
proof.
\end{proof}

\begin{remark}[Continuity constant and $\delta_K$ dependence]
  The continuity constant $C_B$ in Lemma~\ref{lem:dg-coerc} is given
  by
  \begin{equation}
    C_B^2 =
    \epsilon
    +
    \max_{K \in \mathcal{T}} \delta_K
    +
    C_{\operatorname{tr}}^2 \qp{ 1 + \Norm{ \mathscr S }_{\leb\infty(\W_E)} }^2.
  \end{equation}
 
  In the case of the Bragg--Kleeman model \eqref{eq:BraggKleeman}, the
  stopping power satisfies
  \begin{equation}
    \mathscr S(E) = \frac{1}{\alpha p} E^{1 - p}, \quad
    \Norm{ \mathscr S }_{\leb\infty(\W_E)} \sim E_{\min}^{1 - p},   
  \end{equation}
  so $C_B$ may still become large as $E_{\min} \to 0$ when $p > 1$,
  reflecting the degeneracy of the model at low energy. We take
  $E_{\min} > 0$ fixed throughout this work.
\end{remark}

\begin{theorem}[Quasi-optimality]
  \label{thm:cea}
  Let $\psi \in \sobh{1}(\W)$ be the solution to the continuous
  problem \eqref{eq:pde_variational_form}, and let $\psi_h \in \fes$
  solve the discrete problem \eqref{eq:SUPGsoln}. Then
  \begin{equation}
    \enorm{ \psi - \psi_h }
    \leq
    \qp{1 + C_B} \inf_{v_h \in \fes} \enorm{ \psi - v_h }_*
  \end{equation}
  where $C_B > 0$ is the continuity constant from
  Lemma~\ref{lem:dg-coerc}.
\end{theorem}
\begin{proof}
  Let $v_h \in \fes$ be arbitrary and define the error $e_h = \psi_h -
  v_h$. By consistency of the method (Lemma~\ref{lem:consistency}) and
  Galerkin orthogonality,
  \begin{equation}
    \mathcal{A}(\psi_h - v_h, \psi_h - v_h)
    +
    \mathcal{B}_h(\psi_h - v_h, \psi_h - v_h)
    =
    \mathcal{A}(\psi - v_h, \psi_h - v_h)
    +
    \mathcal{B}_h(\psi - v_h, \psi_h - v_h).
  \end{equation}
  Applying the coercivity estimate from Lemma~\ref{lem:dg-coerc},
  \begin{equation}
    \enorm{ \psi_h - v_h }^2
    \leq
    \mathcal{A}( \psi_h - v_h, \psi_h - v_h )
    +
    \mathcal{B}_h( \psi_h - v_h, \psi_h - v_h ).
  \end{equation}
  Applying the continuity bound from Lemma~\ref{lem:dg-coerc},
  \begin{equation}
    \mathcal{A}(\psi - v_h, \psi_h - v_h)
    +
    \mathcal{B}_h(\psi - v_h, \psi_h - v_h)
    \leq
    C_B \enorm{ \psi - v_h }_* \enorm{ \psi_h - v_h }.
  \end{equation}
  Dividing both sides, we obtain
  \begin{equation}
    \enorm{ \psi_h - v_h }
    \leq
    C_B \enorm{ \psi - v_h }_*.
  \end{equation}
  Since this holds for all $v_h \in \fes$, we conclude
  \begin{equation}
    \enorm{ \psi - \psi_h }
    \leq
    \inf_{v_h \in \fes} \enorm{ \psi - v_h } + \enorm{ \psi_h - v_h }
    \leq
    \qp{1 + C_B} \inf_{v_h \in \fes} \enorm{ \psi - v_h }_*.
  \end{equation}
  The result follows after absorbing the prefactor into the constant.
\end{proof}

\begin{corollary}[Best approximation]
  \label{cor:bestApprox}
  Let $\psi \in \sobh{s+1}(\W)$ for some $s \geq 1$, and let $\psi_h
  \in \fes$ solve the discrete problem \eqref{eq:SUPGsoln}. Suppose the stabilisation parameter $\delta_K$ is chosen according to
  \begin{equation}
    \delta_K = \frac{h_K}{2 \qp{\norm{\vec{\omega}} + \norm{\Pi_0 \mathscr{S}(E)}}},
  \end{equation}
  where $\Pi_0$ denotes the $\leb{2}$ projection onto elementwise
  constants. Then the following a priori estimate holds
  \begin{equation}
    \enorm{\psi - \psi_h}
    \leq
    \qp{
      C_1 \epsilon^{1/2}
      h^{\min\qp{p, s}}
      +
      C_2 
      h^{\min\qp{p + 1/2, s}}
    }
    \norm{\psi}_{\sobh{s+1}(\W)},
  \end{equation}
  where $C_1, C_2 > 0$ are constants independent of $h$.
\end{corollary}

\begin{proof}
  Let $I_h \psi \in \fes$ denote the Lagrange interpolant of
  $\psi$. Applying Theorem~\ref{thm:cea} yields
  \begin{equation}
    \enorm{ \psi - \psi_h } \leq \qp{1+C_B} \enorm{ \psi - I_h \psi }_*.
  \end{equation}
  To estimate each component of the error, we apply standard
  interpolation estimates for $I_h \psi$ assuming $\psi \in
  \sobh{s+1}(\W)$. First, for the angular diffusion term
  \begin{equation}
    \Norm{ \nabla_{\vec \omega} (\psi - I_h \psi) }_{\leb2(\W)}
    \leq C h^{\min(k, s)} \Norm{ \psi }_{\sobh{s+1}(\W)}.
  \end{equation}
  Next, for the stabilisation term involving the transport operator  
  \begin{equation}
    \qp{
      \sum_{K\in\mathcal T}\Norm{ \delta_K^{1/2} \cL(\psi - I_h \psi) }_{\leb2(K)}^2
      }^{1/2}
    \leq
    C h^{\min(k + 1/2, s)} \Norm{ \psi }_{\sobh{s+1}(\W)}.
  \end{equation} 
  This rate reflects the half-order gain due to the use of local
  $\delta_K \sim h_K$ stabilisation.

  For the boundary term we apply a trace inequality
  \begin{equation}
    \Norm{ \psi - I_h \psi }_{\leb2(\Gamma^+)}^2
    \leq
    C h^{2 \min(k + 1/2, s)} \Norm{ \psi }_{\sobh{s+1}(\W)}^2.
  \end{equation}
  Finally, for the dual norm contribution
  \begin{equation}
    \sum_{K \in \mathcal{T}} \delta_K^{-1} \Norm{ \psi - I_h \psi }_{\leb2(K)}^2
    \leq
    C\sum_K h_K^{-1} h_K^{2 \min(k+1, s+1)} \Norm{ \psi }_{\sobh{s+1}(K)}^2
    \leq
    Ch^{2 \min(k + 1/2, s)} \Norm{ \psi }_{\sobh{s+1}(\W)}^2.
  \end{equation}
  Collecting all terms, we find
  \begin{equation}
    \enorm{ \psi - \psi_h }
    \leq
    C\epsilon^{1/2} h^{\min(k, s)} \norm{ \psi }_{\sobh{s+1}(\W)}
    +
    C h^{\min(k + 1/2, s)} \norm{ \psi }_{\sobh{s+1}(\W)},
  \end{equation}
  as claimed.
\end{proof}

\begin{remark}[Error rates and the role of $\epsilon$]  
  Notice, if $\epsilon = 0$, corresponding to the absence of angular
  diffusion, the bilinear form remains coercive with respect to
  $\enorm{\cdot}$. However, the problem becomes purely hyperbolic and
  a stronger result is obtained, that is, the error rate improves to
  $\Oh(h^{p+1/2})$.
\end{remark}

\section{Computation of Absorbed Dose}
\label{sec:dose}

The absorbed dose $\cD(\vec x)$ represents the energy deposited per unit
mass at a given spatial position $\vec{x} \in \W_{\vec{x}}$. It is
computed as
\begin{equation}\label{eq:absorbed_dose}
  \cD(\vec x)
  =
  \int_{E_{\text{min}}}^{E_{\text{max}}} \frac{S(E)}{\rho(\vec{x})} \psi(\vec{x}, E) \d E,
\end{equation}
where $\rho(\vec{x})$ is the mass density of the medium at position
$\vec{x}$. 

The absorbed dose is an important output in proton beam therapy. The
nature of proton interactions with matter leads to the Bragg peak,
where protons deposit the majority of their energy near the end of
their range. Accurate computation of dose ensures that the tumour
receives the prescribed energy deposition while minimising exposure to
healthy tissues.

To approximate the absorbed dose numerically, we consider the discrete
fluence $\psi_h$ and define the approximate dose $\cD_h$ via quadrature
in energy. Several approaches are possible depending on how the result
is projected into a spatial finite element space. Let $\{\xi_q\}_{q=1}^{Q_E}$
and $\{w_q\}_{q=1}^{Q_E}$ be quadrature nodes and weights in energy.

\subsection{Galerkin projection}

A natural approach is to define $\cD_h$ by spatial projection of the
integrand in \eqref{eq:absorbed_dose} onto the finite element space
$\fes_{\vec{x}}$. Given the discrete fluence $\psi_h$, we define $\cD_h
\in \fes_{\vec{x}}$ by the variational formulation
\begin{equation}\label{eq:Dh_projection}
  \int_{\W_{\vec{x}}} \cD_h(\vec x) v_h(\vec x) \, \d \vec{x}
  =
  \int_{\W_{\vec{x}}}
  \sum_{q=1}^{Q_E} w_q \frac{S(\xi_q)}{\rho(\vec{x})} \psi_h(\vec{x}, \xi_q)
  v_h(\vec x)
  \d \vec{x}
  \Foreach v_h \in \fes_{\vec{x}}.
\end{equation}
That is $\cD_h$ is the $\leb{2}$-projection of the energy-integrated
quantity onto the spatial finite element space $\fes_{\vec{x}}$.

\begin{lemma}[Error estimate for Galerkin dose projection]
  \label{lem:galerkin}
  Let $\psi \in \sobh{s+1}(\W)$ for some $s \geq 1$, and let $\psi_h
  \in \fes$ be the finite element approximation. Let $\cD$ denote the
  exact dose defined by \eqref{eq:absorbed_dose}, and let $\cD_h \in
  \fes_{\vec x}$ be the Galerkin projection defined in
  \eqref{eq:Dh_projection} using a quadrature rule of order $q_E$ in
  energy with step size $h_E$. Then
  \begin{equation}
    \Norm{ \cD - \cD_h }_{\leb2(\W_{\vec x})}
    \leq
    C_1 h_E^{\min(q_E, s)} \Norm{ \psi }_{\sobh{s+1}(\W) }
    +
    C_2 h^{\min(k, s)} \Norm{ \psi }_{\sobh{s+1}(\W) }.
  \end{equation}  
\end{lemma}

\begin{proof}
  Define the auxiliary quadrature-integrated exact dose as
  \begin{equation}
    \cD^*(\vec x)
    :=
    \sum_{q=1}^{Q_E} w_q \frac{S(\xi_q)}{\rho(\vec x)} \psi(\vec x, \xi_q).
  \end{equation}
  Now decompose the error as
  \begin{equation}
    \Norm{ \cD - \cD_h }_{\leb2(\W_{\vec x})}
    \leq
    \Norm{ \cD - \cD^* }_{\leb2(\W_{\vec x})}
    +
    \Norm{ \cD^* - \cD_h }_{\leb2(\W_{\vec x})}.
  \end{equation}
  The first term is the quadrature error. Since $\psi \in
  \sobh{s+1}(\W)$, the map $E \mapsto \psi(\vec x, E)$ is smooth for
  each fixed $\vec x$ and standard quadrature theory yields
  \begin{equation}
    \Norm{ \cD - \cD^* }_{\leb2(\W_{\vec x})}
    \leq
    C_1 h_E^{\min(q_E, s)} \Norm{ \psi }_{\sobh{s+1}(\W) }.
  \end{equation}
  The second term is the projection error in space. Since $\psi$ is
  smooth $\cD^*$ is smooth in $\vec x$ and 
  \begin{equation}
    \Norm{ \cD^* - \cD_h }_{\leb2(\W_{\vec x})}
    \leq
    C h^{\min(k, s)} \Norm{ \psi }_{\sobh{s+1}(\W) }.
  \end{equation}
  Combining both gives the stated bound.
\end{proof}

\begin{remark}[Failure of positivity]
  Although the Galerkin projection is variationally consistent and
  convergent, it does not preserve nonnegativity. Even if $\psi_h(\vec
  x, E) \geq 0$ for all $(\vec x, E) \in \W$, and $S(E), \rho(\vec x)
  > 0$, the projection onto $\fes_{\vec x}$ may produce $\cD_h(\vec
  x_i) < 0$ at some nodes due to sign changes in the basis functions.
  
  In particular, even when $\psi_h$ is computed using a nodally bound
  preserving method, the dose $\cD_h$ does not inherit this property
  unless additional constraints are enforced. This limitation makes
  the Galerkin approach unsuitable in contexts where pointwise
  nonnegativity of the dose is required, for example, as part of a
  constraint in inverse problems or optimisation routines.
\end{remark}

\subsection{Elementwise constant projection}

A robust and positivity-preserving alternative is to compute $\cD_h$ as
a piecewise constant function. For each element $K \in
\mathcal{T}_{\vec{x}}$, define
\begin{equation}
  \cD_h(\vec x)\vert_K
  :=
  \frac{1}{|K|}
  \int_K
  \sum_{q=1}^{Q_E} w_q \frac{S(\xi_q)}{\rho(\vec{x})} \psi_h(\vec{x}, \xi_q)
  \, \d \vec{x}.
\end{equation}

\begin{lemma}[Error estimate for elementwise constant projection]
  Let $\psi \in \sobh{2}(\W)$ and let $\cD$ be the exact dose
  \eqref{eq:absorbed_dose}. Let $\cD_h$ be the elementwise constant
  approximation defined above, using a quadrature rule of order $q_E
  \geq 1$ in energy. Then
  \begin{equation}
    \Norm{ \cD - \cD_h }_{\leb2(\W_{\vec x})}
    \leq
    C h \Norm{ \psi }_{\sobh{2}(\W)},
  \end{equation}
  where $C$ depends on bounds for $S$ and $\rho$ but is independent of
  $h$.
\end{lemma}

\begin{proof}
  The proof follows the same lines as Lemma \ref{lem:galerkin} noting
  that the piecewise constant structure limits accuracy.
\end{proof}

\begin{remark}[Positivity and stability]
  The elementwise constant projection preserves positivity at the
  level of elements. If $\psi_h \geq 0$, given $S, \rho > 0$, then
  $\cD_h \geq 0$ on each $K \in \mathcal{T}_{\vec x}$. This follows
  directly from the positivity of the integrand and weights in the
  quadrature rule. In addition, the projection is stable in $\leb2$.

  While it lacks higher-order accuracy in space, it is robust and easy
  to implement and maybe sufficient for many visualisation and dose
  aggregation tasks where smoothness is not critical.
\end{remark}

\subsection{Variational inequality projection}

To remain consistent with the nodal positivity framework, a natural
approach amenable to higher order approximations is to define the
discrete dose as the solution of a constrained variational inequality
that preserves nonnegativity at nodal degrees of freedom. Let
\begin{equation}
  K_{\vec{x}} :=
  \ensemble{ v_h \in \fes_{\vec{x}} }{ v_h(\vec x_i) \geq 0 \text{ for all nodes } \vec x_i }.
\end{equation}
Then the variational inequality formulation reads: Find $\cD_h \in
K_{\vec{x}}$ such that
\begin{equation}
  \int_{\W_{\vec{x}}} \cD_h(\vec x) (v_h(\vec x) - \cD_h(\vec x)) \d \vec x
  \geq
  \int_{\W_{\vec{x}}}
  \sum_{q=1}^{Q_E} w_q \frac{S(\xi_q)}{\rho(\vec{x})} \psi_h(\vec{x}, \xi_q)
  (v_h(\vec x) - \cD_h(\vec x)) \d \vec{x}
  \Foreach v_h \in K_{\vec{x}}.
\end{equation}
This formulation ensures that the discrete dose satisfies the same
nodal bound-preserving structure as the underlying fluence
$\psi_h$. In particular, assuming the quadrature is chosen to have
positive weights, the resulting dose $\cD_h$ satisfies $\cD_h(\vec
x_i) \geq 0$ at all nodes, regardless of the sign of the basis
functions.

\section{Numerical Experiments}
\label{sec:numerics}

In this section, we assess the accuracy, stability and computational
efficiency of the proposed scheme across different proton transport
scenarios. The numerical experiments begin with a benchmark comparison
against an analytical solution for a pristine Bragg peak, providing a
reference for the deterministic model. We then investigate the role of
mesh adaptivity in improving solution accuracy and controlling
numerical artefacts, particularly in regions with steep gradients. The
integration of Coulomb scattering is examined to quantify its impact
on fluence distributions. Beyond accuracy considerations, we analyse
the computational complexity of the scheme, assessing the efficiency
of the method. Finally, we extend to more physically realistic setups,
demonstrating its robustness in scenarios relevant to proton therapy
applications. The numerical results in this work were generated using
Firedrake \cite{FiredrakeUserManual}. We utilise PETSc
\cite{petsc-efficient} to solve the variational inequalities using a
Reduced-Space Active Set Method, or LU factorisation for the standard
variational formulations. When needed, we use Netgen meshes in
Firedrake to enable adaptive mesh refinements \cite{Betteridge2024}.

\subsection{Example 1: Benchmark - Pristine Bragg peak with $\epsilon = 0$}

To assess the accuracy of the deterministic model, we consider a
pristine Bragg peak simulation as a computational benchmark. A
monoenergetic proton beam with an initial energy of $E_0 = 62 \,
\text{MeV}$ and a fluence of $1.21 \, \text{gigaprotons/cm}^2$ is used
as the input beam. The analytical model from \cite{ashby2025efficient}
employs standard Bragg-Kleeman parameter values for water ($ \alpha =
2.2 \times 10^{-3}$, $p = 1.77$) as reported in
\cite{bortfeld1997analytical}. To account for energy spread in the the
standard deviation of the proton energy spectrum is set to $\delta
E_0$, with $\delta = 0.01$.

The boundary condition for the analytical model is defined as
\begin{equation}\label{eq:gigaproton_spectrum}
  \psi(\vec 0, E)
  =
  1.21 \times 10^9 \times C \exp\left(-\frac{(E - E_0)^2}{2 \delta^2 E_0^2}\right),
\end{equation}
where $C$ is a normalisation constant ensuring that the integral of
the spectrum matches the total fluence of $1.21 \,
\text{gigaprotons/cm}^2$. The exact solution for the fluence as a
function of depth and energy for $\epsilon=0$, is given by
\begin{equation}\label{eq:defn_fluence}
  \psi^{exact}(\vec{x}, E) =
  \left(E^p + \frac{\vec{\omega} \cdot \vec{x}}{\alpha}\right)^{\frac{1-p}{p}}
  g\left(\left(E^p + \frac{\vec{\omega} \cdot \vec{x}}{\alpha}\right)^{\frac{1}{p}}\right)
  E^{p-1},
\end{equation}
where $g(E)$ is defined by the boundary condition in
\eqref{eq:gigaproton_spectrum} as
\begin{equation}
  g(E) = 1.21 \times 10^9 \times C \exp\left(-\frac{(E - E_0)^2}{2 \epsilon^2 E_0^2}\right).
\end{equation}
We compare the numerical approximations obtained using the vanilla
SUPG method (\ref{eq:SUPGsoln}) and the positivity-preserving scheme
(\ref{eq:discrete_inequality}). Figure \ref{fig:ex1supg} shows the
results for the SUPG method, while Figure \ref{fig:ex1vi} presents the
corresponding results for the positivity-preserving formulation. Each
plot includes the inflow energy boundary conditions, the space-energy
fluence and the absorbed dose.

The SUPG method exhibits spurious oscillations in the fluence, which
propagate throughout the domain due to its lack of strict positivity
preservation. These artefacts are particularly pronounced in regions
with steep energy gradients, where numerical dispersion causes
unphysical undershoots and overshoots in the cross-wind direction. In
contrast, the positivity-preserving scheme eliminates these
oscillations, providing a sharp and physically consistent resolution
of the fluence distribution. The improved stability of the scheme
ensures that the dose computation remains well-behaved, preventing
nonphysical negative values in low-dose regions.

\begin{figure}[h!]
    \centering
    \includegraphics[width=0.6\linewidth]{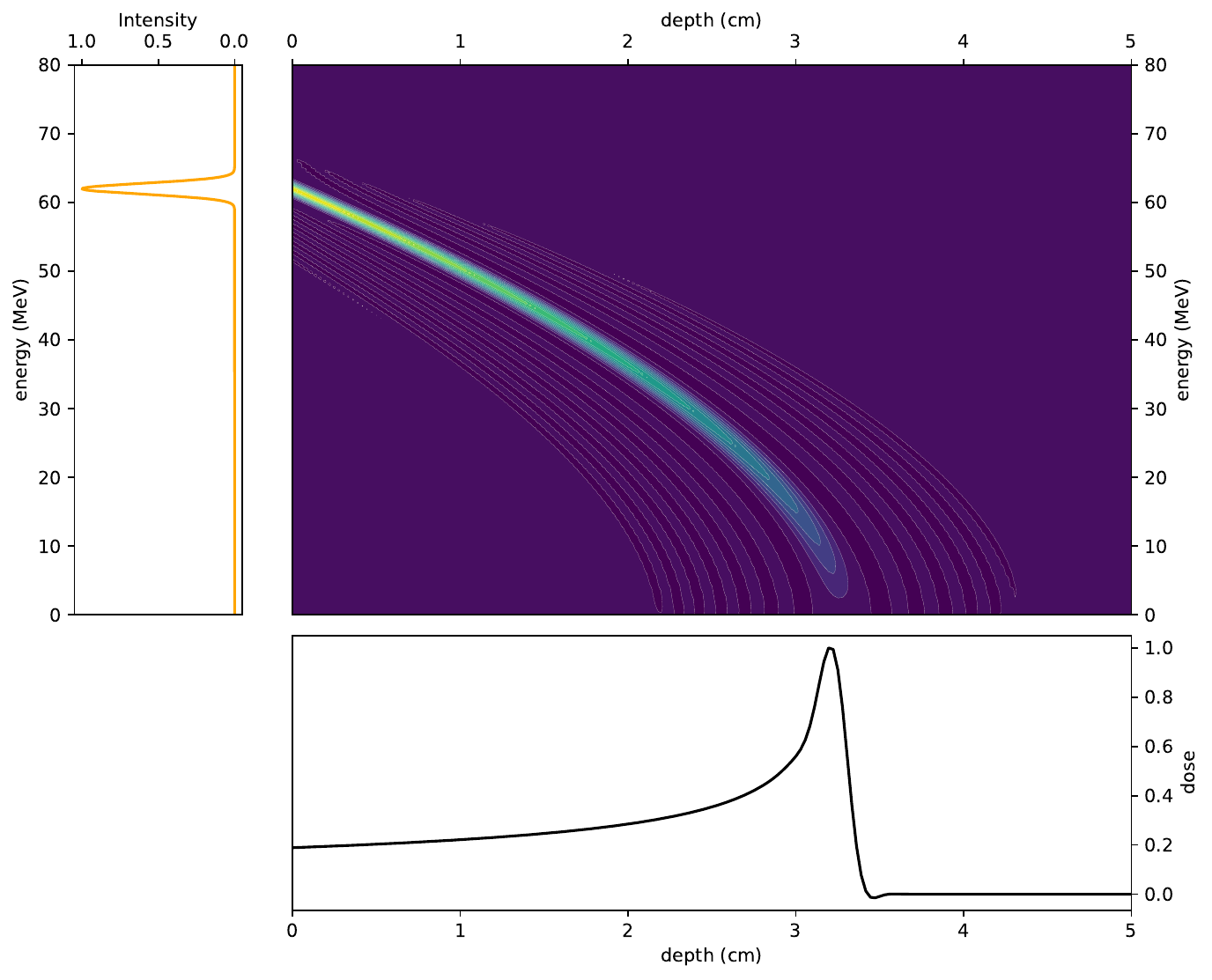}
    \caption{ \label{fig:ex1supg} Numerical solution obtained using
      the SUPG method (\ref{eq:SUPGsoln}). The plots display the
      inflow energy boundary conditions, the space-energy fluence and
      the absorbed dose. Spurious oscillations appear in the fluence,
      propagating throughout the domain due to the lack of strict
      positivity preservation. Notice this results in a slight
      negative dose after the Bragg peak.}
\end{figure}

\begin{figure}[h!]
    \centering
    \includegraphics[width=0.6\linewidth]{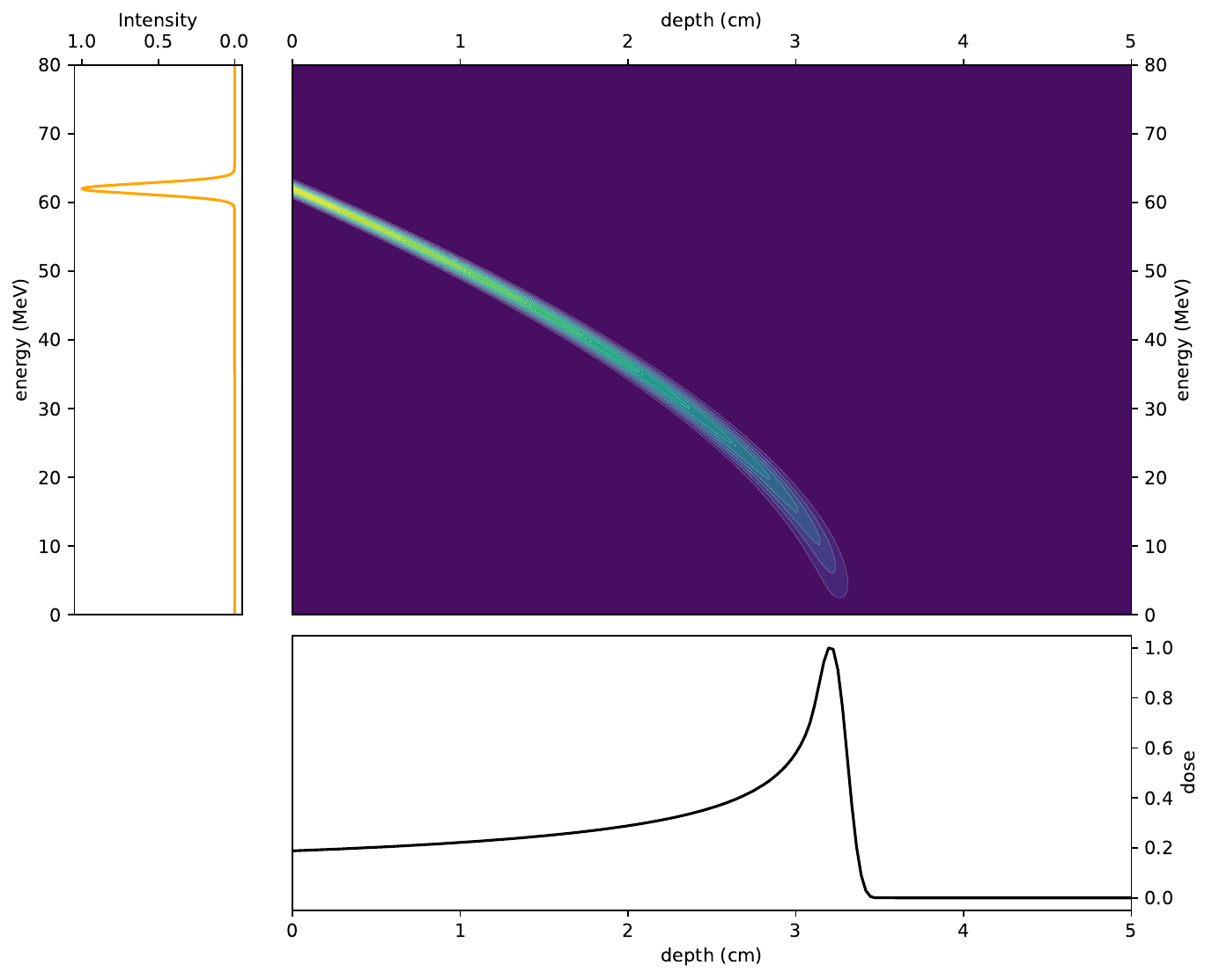}
    \caption{ \label{fig:ex1vi} Numerical solution obtained using the
      positivity-preserving scheme (\ref{eq:discrete_inequality}). The
      plots display the inflow energy boundary conditions, the
      space-energy fluence, and the absorbed dose. Unlike the SUPG
      method, this scheme eliminates oscillations and provides a
      sharper, more physically consistent resolution of the fluence
      distribution.}
\end{figure}

\clearpage
\subsection{Example 2: Adaptive mesh refinement}

For a 'monoenergetic' \footnote{ In practice, the beam has a narrow
but smooth energy spectrum (e.g., a Gaussian centred at $E_0$). From a
mathematical perspective, a truly monoenergetic beam would require a
Dirac delta in energy, which is not square-integrable and leads to
non-classical inflow boundary data.  } beam, the behaviour of the
fluence is highly localised, motivating the construction of an a
posteriori indicator to drive mesh adaptivity. To that end, consider a
sequence of nested finite element spaces
$\qp{\fes_i}_{i\in\naturals_0}$ with $\fes_0 \subset \fes_1 \subset
\cdots$. Let $\fes_0$ denote a prescribed coarse mesh, then, over each
space, consider the solution $\psi_h^i \in \fes_i$ satisfying
(\ref{eq:discrete_inequality}).

We define an a posteriori error indicator that combines the transport
residual with angular diffusion jump terms across element
interfaces. For each element $K \in \mathcal{T}$, define
\begin{equation}
  \label{eq:indicator}
  \eta_K^2
  :=
  \Norm{ \cL(\psi_h) }_{\leb2(K)}^2
  +
  \epsilon
  \sum_{e \subset \partial K}
  h_e \,
  \Norm{ \llbracket \nabla \psi_h \cdot \vec{\omega}^\perp \rrbracket }_{\leb2(e)}^2,
\end{equation}
where $h_e$ is the length of the face $e$. This indicator reflects
both the misalignment of $u_h$ with transport characteristics and the
lack of regularity across element interfaces introduced by angular
diffusion.

Using a maximum marking strategy given in Algorithm
\ref{alg:adaptive_transport}, elements satisfying
\begin{equation}
  \eta_K \geq \theta \max_{K' \in \mathcal{T}} \eta_{K'},
\end{equation}
with $\theta = 0.01$, are refined to construct the next finite element
space $\fes_{i+1}$.

\begin{algorithm}[h!]
  \caption{Adaptive mesh refinement for proton transport}
  \label{alg:adaptive_transport}
  \begin{algorithmic}[1]
    \Require Initial mesh $\mathcal{T}_0$, tolerance parameter $\theta \in (0,1]$, maximum refinement levels $N_{\max}$.
    \State Set $i = 0$.
    \While{$i < N_{\max}$}
      \State \textbf{Solve.} Find $\psi_h^i \in \fes_i$ solving (\ref{eq:discrete_inequality}).
      \State \textbf{Estimate.} Compute elementwise error indicator (\ref{eq:indicator}).
      \State \textbf{Mark} a set of elements
      \begin{equation*}
        \mathcal{M}_i = \{ K \in \mathcal{T}_i \mid \eta_K \geq \theta \max_{K' \in \mathcal{T}_i} \eta_{K'} \}.
      \end{equation*}
      \State \textbf{Refine.} Construct a new finite element space $\fes_{i+1}$ by refining elements in $\mathcal{M}_i$.
      \State \textbf{Update} the mesh: $\mathcal{T}_{i+1} \gets \text{refined mesh}$.
      \State Set $i \gets i+1$.
    \EndWhile
  \end{algorithmic}
\end{algorithm}

Figure~\ref{fig:ex1-adaptive} illustrates the effect of adaptive mesh
refinement driven by the a posteriori indicator $\eta_K$ defined in
\eqref{eq:indicator}. The left panel shows the refined mesh and
corresponding fluence $\psi_h$ computed using
Algorithm~\ref{alg:adaptive_transport}, while the right panel reports
convergence of the energy norm error under uniform and adaptive
refinement. The adaptive strategy achieves comparable accuracy with
significantly fewer degrees of freedom.

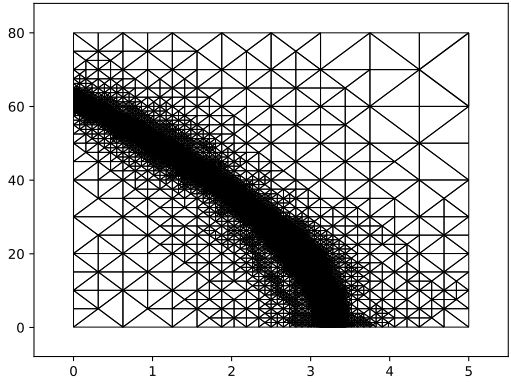
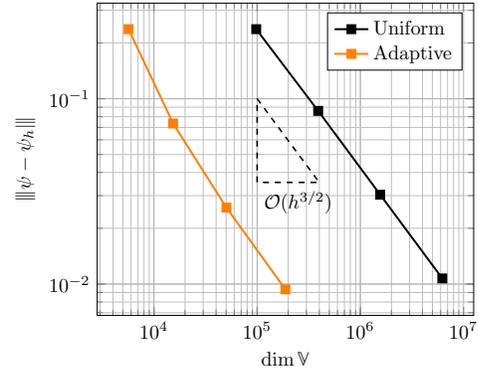
\begin{figure}[h!]
  \centering
  \begin{subfigure}[b]{0.45\textwidth}
    \includegraphics[trim=20 20 20 20, clip, width=\linewidth]{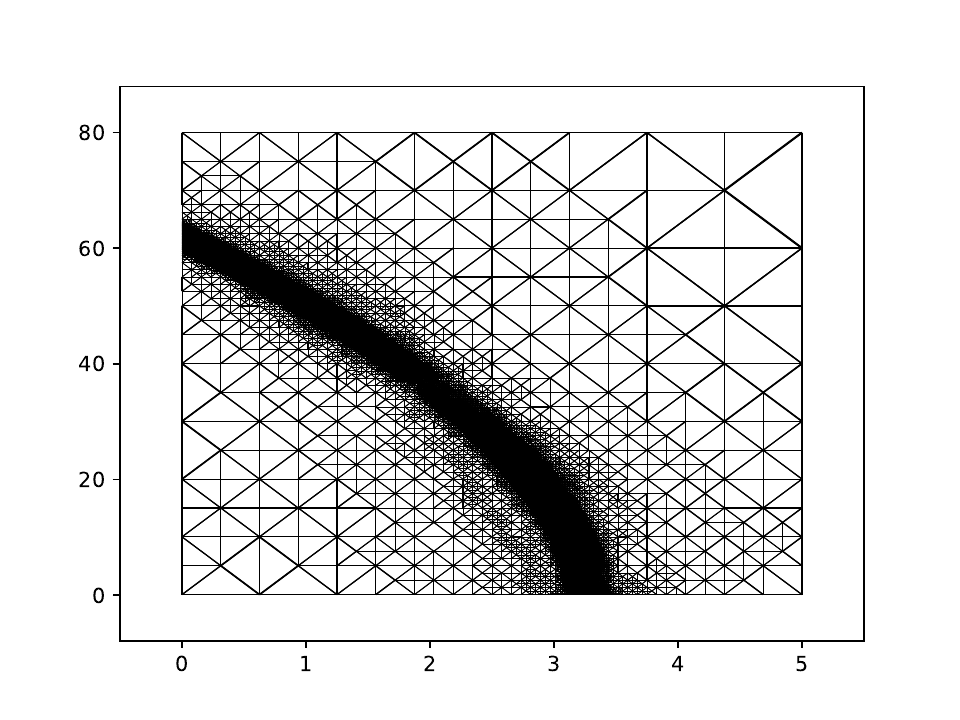}

    \caption{Adaptive mesh.}
    \label{fig:ex1-mesh}
  \end{subfigure}
  \hfill
  \begin{subfigure}[b]{0.5\textwidth}
    \begin{tikzpicture}[scale=.75]
      \begin{axis}[
          width=\linewidth,
          xmode=log, ymode=log,
          grid=both,
          major grid style={black!30},
          xlabel = $\dim{\fes}$,
          ylabel = $\enorm{\psi - \psi_h}$,
          legend style={at={(0.97,0.97)},anchor=north east},
          legend cell align={left}
      ]
      \addplot[solid, mark=square*, mark options={scale=1, solid}, color={black},
               line width=1.0] coordinates {
        (97921, 0.23734511)
        (390241, 0.0859666)
        (1558081, 0.03035232)
        (6226561, 0.01072219)};
      \addlegendentry{Uniform}
      \addplot[solid, mark=square*, mark options={scale=1,solid}, color=orange, 
               line width=1.0] coordinates {
        (5669, 0.23736301)
        (15321, 0.07355093)
        (50350, 0.02584937)
        (187699, 0.00936292)
      };
      \addlegendentry{Adaptive}
      \draw[dashed, thick] 
        (1e5, 0.1) -- (4e5, 0.035355) -- (1e5, 0.035355) -- cycle;
      \node[anchor=south west] at (axis cs:1e5, 0.022) {$\mathcal{O}(h^{3/2})$};
      \end{axis}
    \end{tikzpicture}
    \caption{Convergence under uniform and adaptive refinement.}
    \label{fig:ex1-convergence}
  \end{subfigure}

  \caption{ Adaptive mesh refinement driven by the a posteriori
    indicator $\eta_K$ from \eqref{eq:indicator}.  (a) Refined mesh
    after several iterations of
    Algorithm~\ref{alg:adaptive_transport}.  (b) Convergence of the
    energy norm error $\enorm{\psi - \psi_h}$ under uniform and
    adaptive refinement.  The adaptive method achieves comparable
    accuracy with substantially fewer degrees of freedom.
    \label{fig:ex1-adaptive}  
  }
\end{figure}

In Figure \ref{fig:ex1-dose}, we compare the exact dose distribution
with results from the vanilla SUPG method and the
positivity-preserving SUPG (PP-SUPG) in both uniform and adaptive
refinement settings. All methods appear visually compatible in an
eyeball norm; however, key differences emerge in the details. The
vanilla SUPG method underestimates the magnitude of the Bragg peak and
exhibits unphysical negative dose values beyond the peak. In contrast,
the positivity-preserving scheme eliminates these negative values,
ensuring physically consistent behaviour. The adaptive
positivity-preserving scheme achieves excellent agreement with the
exact solution, accurately capturing both the position and magnitude
of the Bragg peak while using fewer degrees of freedom.

\begin{figure}
  \centering
  \includegraphics[width=0.5\linewidth]{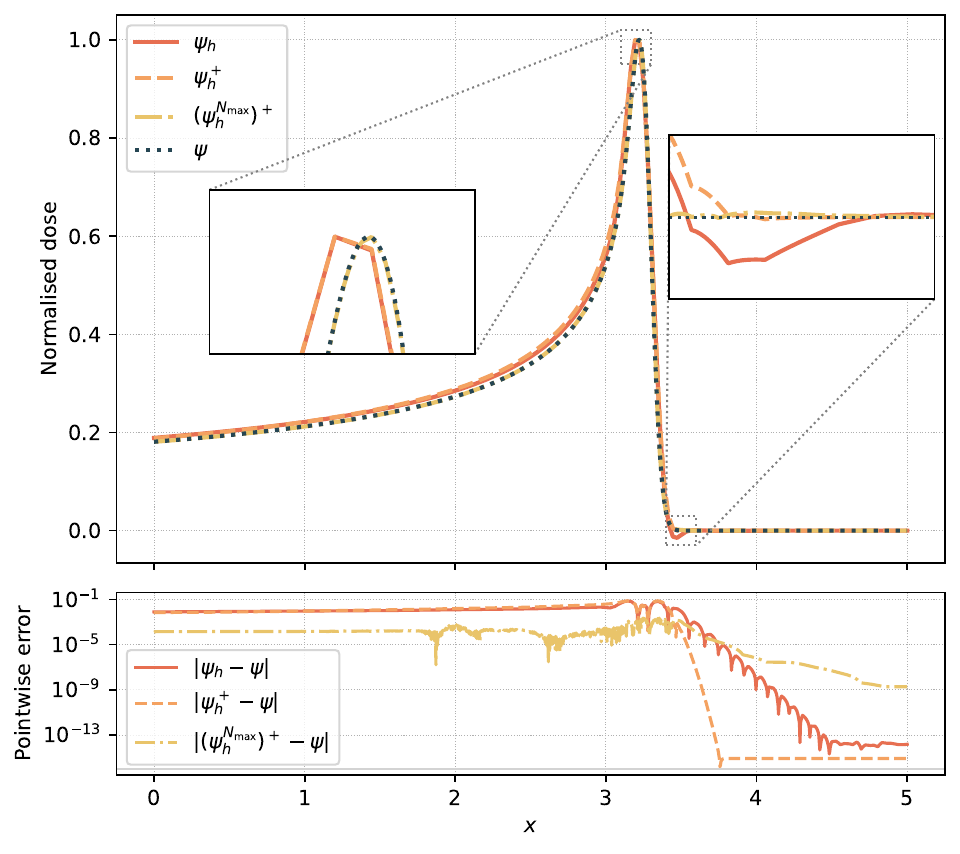}
  \caption{
    \label{fig:ex1-dose}
    Spatial distribution of pointwise error in the absorbed dose,
    comparing three numerical schemes against the exact reference.
    The blue curve corresponds to the standard SUPG method
    ($\approx 9.7 \times 10^5$ DoFs), the red curve to the
    positivity-preserving variational inequality (VI) formulation on
    the same mesh, and the green curve to the VI scheme with adaptive
    refinement ($\approx 5.0 \times 10^5$ DoFs). All use trapezoidal
    quadrature for energy integration. The VI schemes yield
    lower errors throughout the domain, with the adaptive method
    achieving the most accurate result near the Bragg peak while
    maintaining a reduced computational cost.
  }
\end{figure}

\clearpage

\subsection{Example 3: Angular Scattering and Laplace--Beltrami Diffusion}

We now investigate the effect of angular scattering on proton
transport by incorporating a nonzero Laplace--Beltrami coefficient
$\epsilon > 0$ into the model. The underlying operator arises from the
linear Boltzmann transport equation in phase space $(\vec x, \vec
\omega, E) \in \W_{\vec{x}} \times \mathbb{S}^{d-1} \times \W_E$, with
angular scattering described by
\begin{equation}
  \cL(\psi(\vec x, \vec \omega, E))
  =
  \int_{\mathbb{S}^{d-1}} \pi(\vec \omega \cdot \vec \omega') \psi(\vec x, \vec \omega', E) \d \vec \omega'
  -
  \psi(\vec x, \vec \omega, E),
\end{equation}
where $\pi(\vec \omega \cdot \vec \omega')$ is a normalised phase
function. A common choice is the Henyey--Greenstein kernel,
\begin{equation}
  \pi(\vec \omega \cdot \vec \omega') 
  =
  \frac{1 - g^2}{\qp{1 + g^2 - 2g (\vec \omega \cdot \vec \omega')}^{3/2}},
\end{equation}
which becomes sharply forward-peaked as $g \to 1^-$. In this regime,
the operator admits the angular diffusion (Fokker--Planck)
approximation
\begin{equation}
  \cL(\psi)
  \approx
  \epsilon \Delta_{\vec \omega} \psi, \qquad
  \epsilon = \tfrac{1}{2}(1 - g),
\end{equation}
where $\Delta_{\vec \omega}$ denotes the Laplace--Beltrami operator on
the sphere.

We solve the angular diffusion model using the SUPG method described
in \S\ref{sec:supg}, with adaptive mesh
refinement. Figure~\ref{fig:manyeps} shows the dose fields for varying
angular diffusion strengths $\epsilon \in \{0, 0.005, 0.01, 0.1\}$. As
$\epsilon$ increases, angular scattering broadens the fluence and
smooths the Bragg peak, leading to spatially diffuse dose
distributions. These results illustrate the role of angular diffusion
in regulating beam directionality and spatial dose spread.

\begin{figure}[h!]
    \centering
    \begin{subfigure}{0.4\textwidth}
        \includegraphics[width=\linewidth]{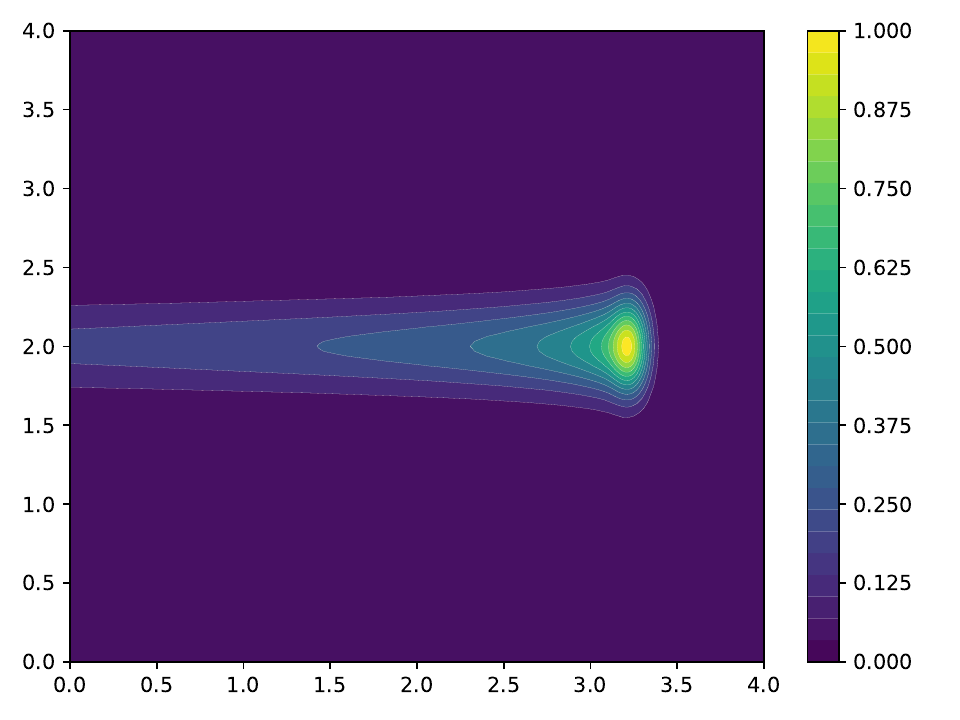}
        \caption{$\epsilon = 0$}
    \end{subfigure}
    \hfill
    \begin{subfigure}{0.4\textwidth}
        \includegraphics[width=\linewidth]{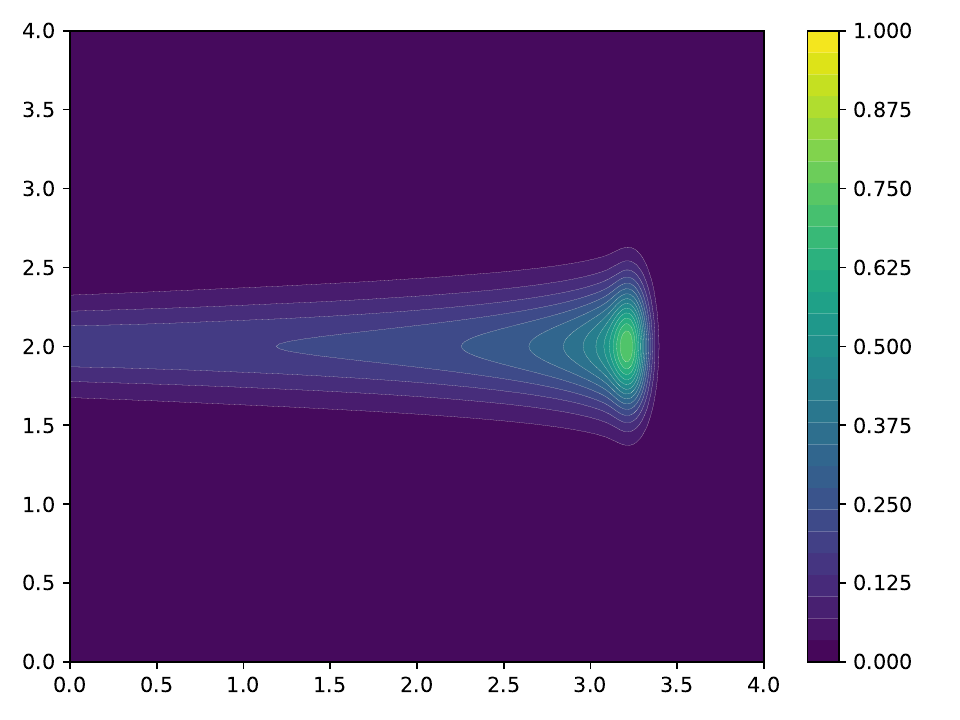}
        \caption{$\epsilon = 0.005$}
    \end{subfigure}
    \begin{subfigure}{0.4\textwidth}
        \includegraphics[width=\linewidth]{Figures/ex3eps_0_01.pdf}
        \caption{$\epsilon = 0.01$}
    \end{subfigure}
    \hfill
    \begin{subfigure}{0.4\textwidth}
        \includegraphics[width=\linewidth]{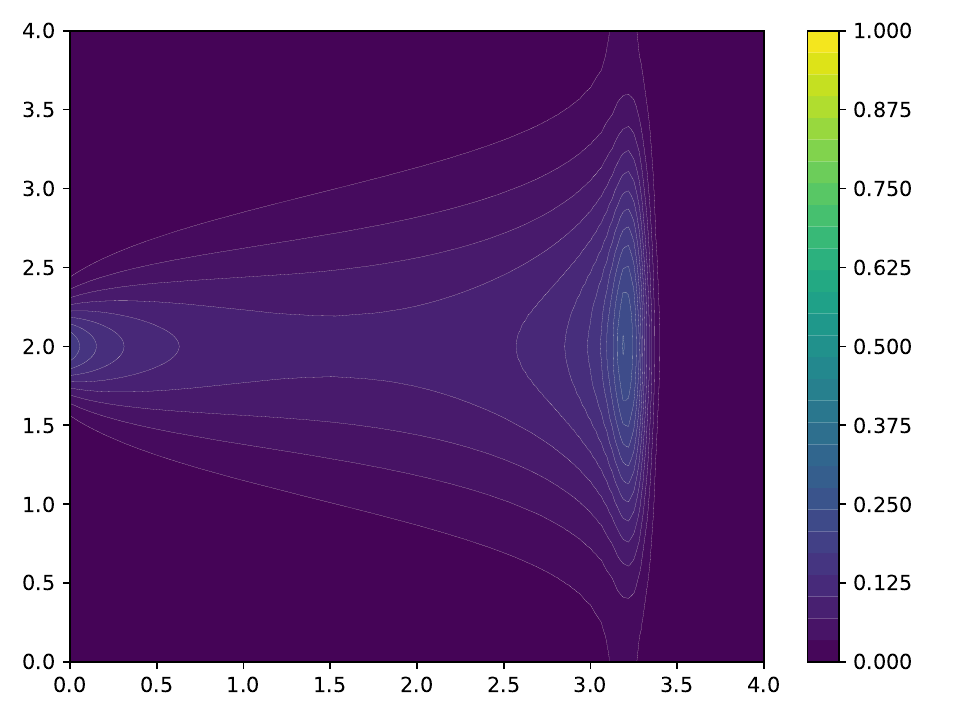}
        \caption{$\epsilon = 0.1$}
    \end{subfigure}
    \caption{
      \label{fig:manyeps}
      Absorbed dose computed using the SUPG scheme with adaptive mesh
      refinement for increasing values of the angular diffusion
      coefficient $\epsilon$. The dose at $\epsilon = 0$ is normalised
      by its maximum value; all others are plotted using the same
      scale for comparison. As $\epsilon$ increases, angular
      scattering smooths the dose distribution, reducing peak
      sharpness and increasing lateral spread.  }
\end{figure}

\begin{figure}[h!]
  \centering
  \begin{subfigure}[b]{0.48\textwidth}
    \includegraphics[width=\linewidth]{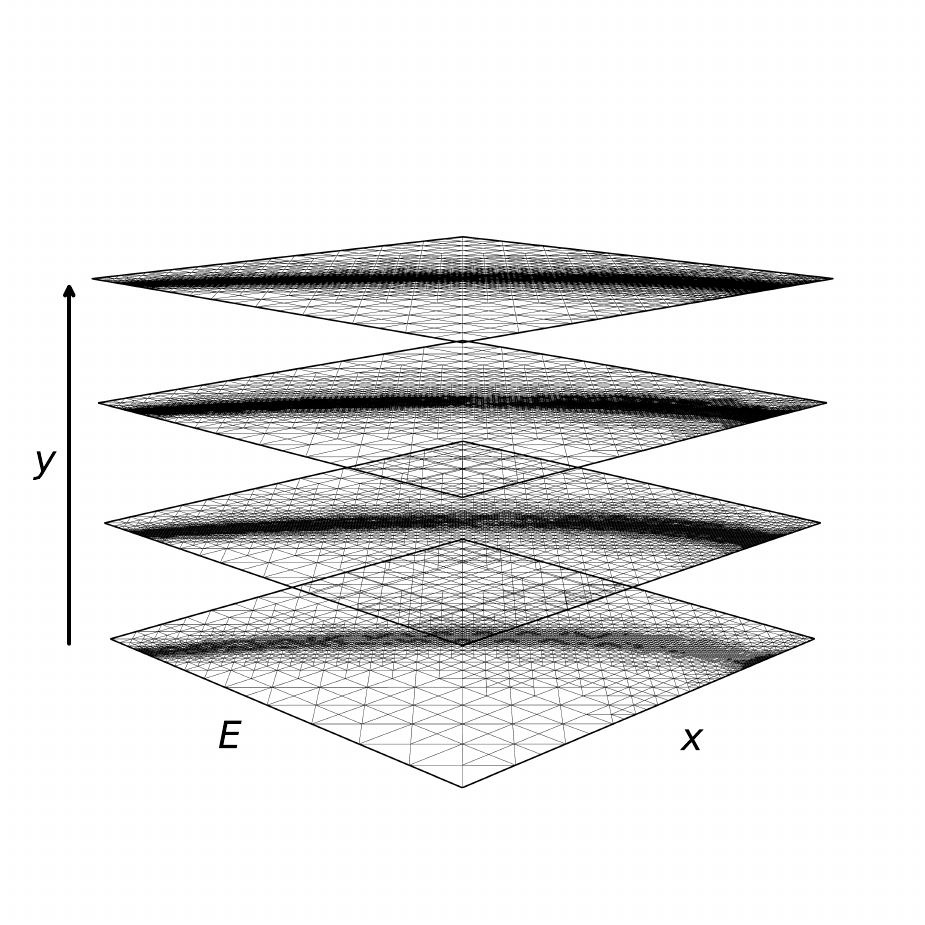}
    \caption{$\epsilon = 0$}
  \end{subfigure}
  \hfill
  \begin{subfigure}[b]{0.48\textwidth}
    \includegraphics[width=\linewidth]{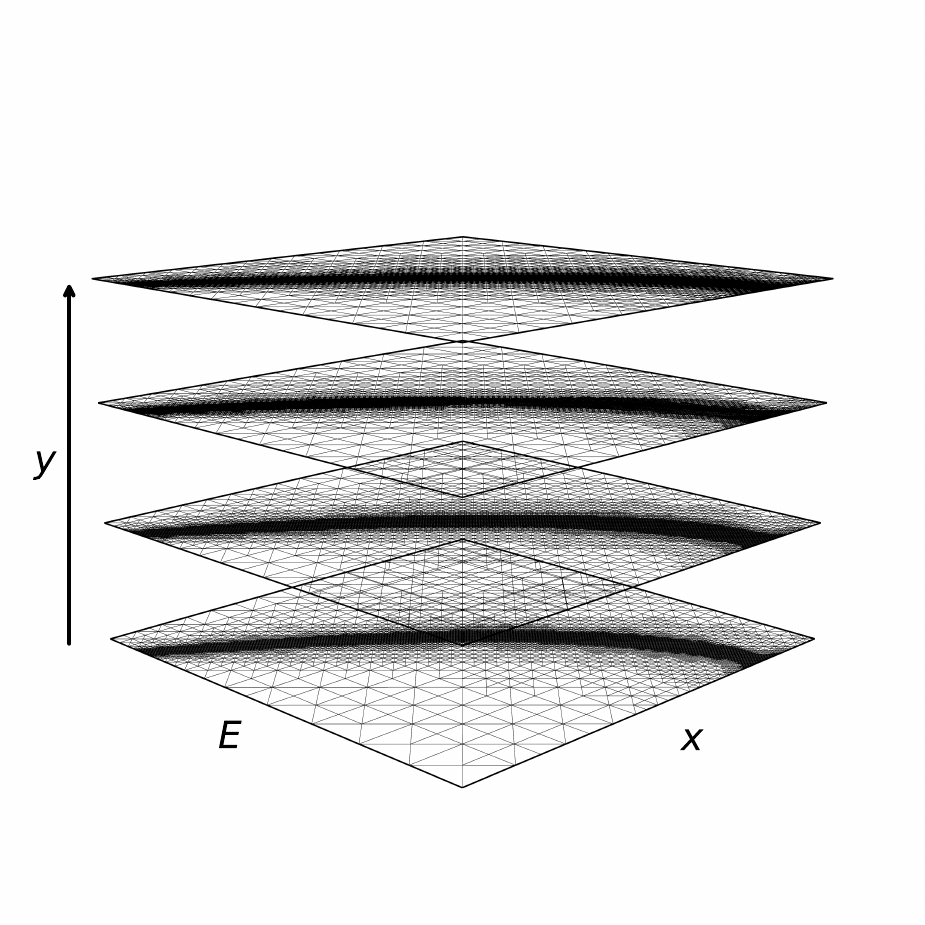}
    \caption{$\epsilon = 0.005$}
  \end{subfigure}
  \vfill
  \begin{subfigure}[b]{0.48\textwidth}
    \includegraphics[width=\linewidth]{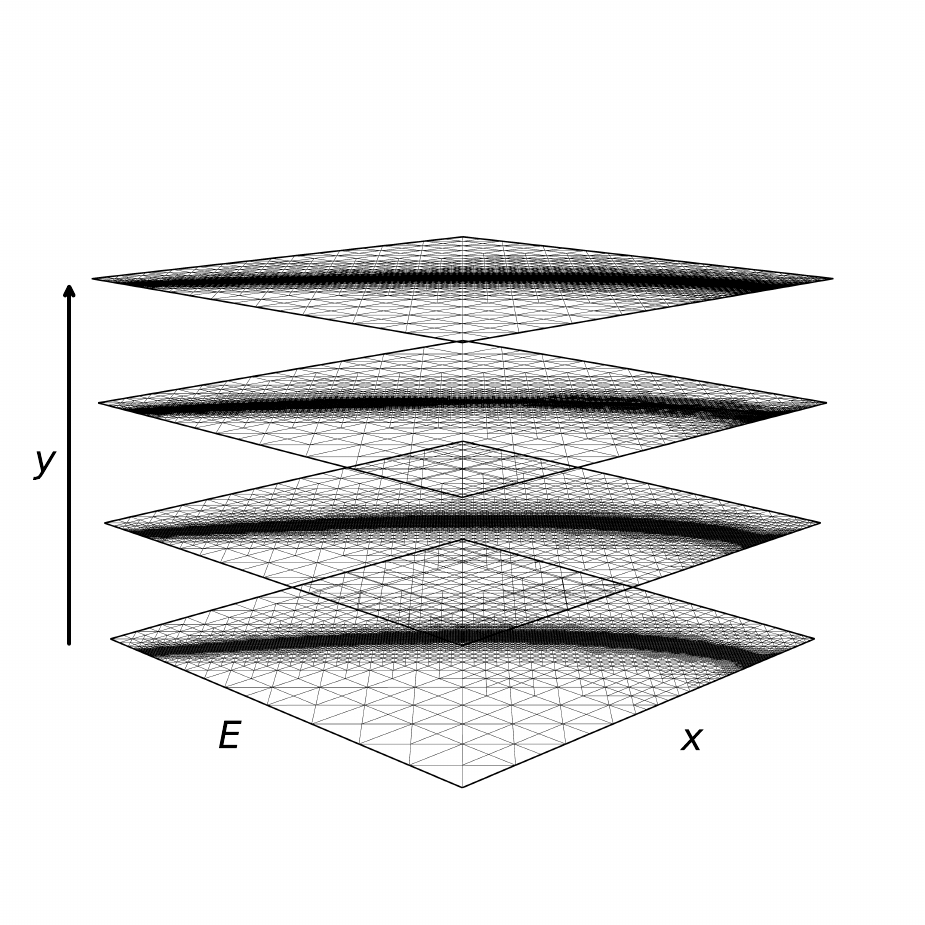}
    \caption{$\epsilon = 0.01$}
  \end{subfigure}
  \hfill
  \begin{subfigure}[b]{0.48\textwidth}
    \includegraphics[width=\linewidth]{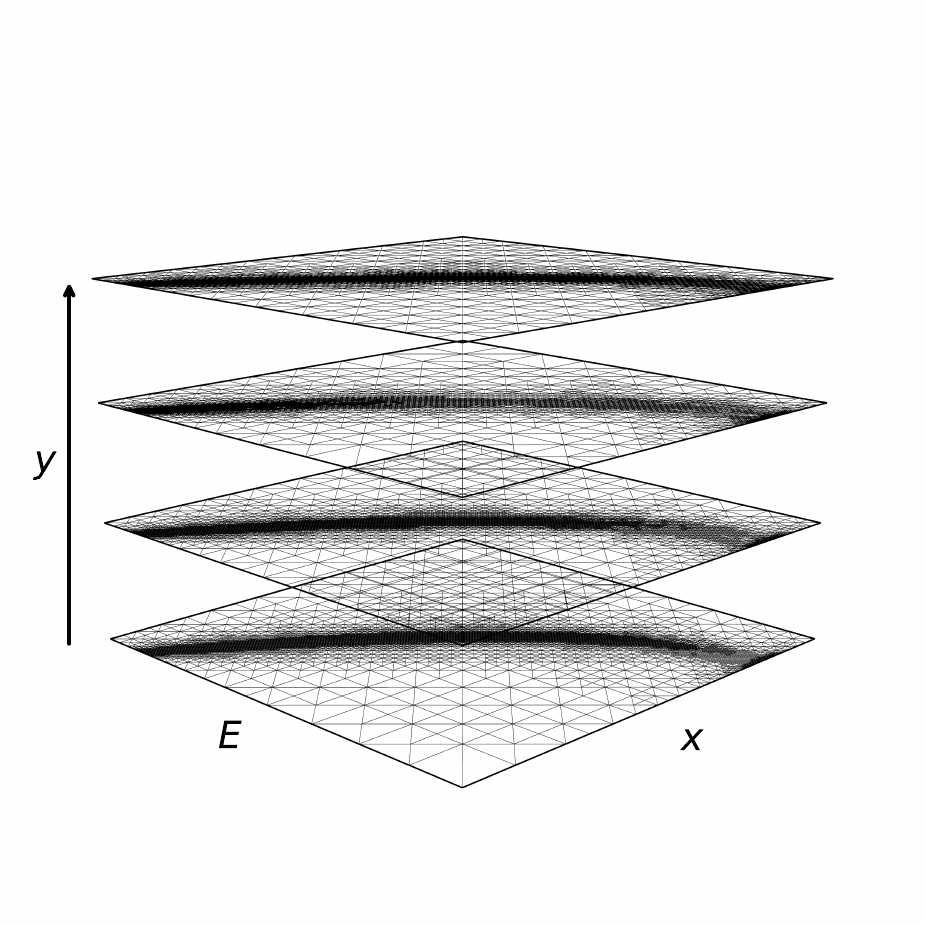}
    \caption{$\epsilon = 0.1$}
  \end{subfigure}
  \caption{
    \label{fig:ex3-stackedmesh}
    Adaptive meshes sliced in energy and rendered as stacked layers
    for varying values of the angular diffusion coefficient $\epsilon
    \in \{0,\, 0.005,\, 0.01,\, 0.1\}$. Finer resolutions follow the
    evolution of the angularly smeared Bragg peak. Compare with the
    dose profiles in Figure~\ref{fig:manyeps} and fluence slices in
    Figure~\ref{fig:fluencemesh}.  }
\end{figure}

\clearpage

\subsection{Example 4: Orbital Case Study with Layered Heterogeneous Media}

We conclude with a case study motivated by ocular proton therapy,
where beam paths pass through anatomically heterogeneous structures
before reaching the posterior orbit. The domain is modelled as a
sequence of planar, homogeneous layers representing key tissue types
encountered along the beam path. An idealised one-dimensional
configuration is shown in Figure~\ref{fig:tissue_depth}. Material
densities are based on representative clinical values, with slight
exaggeration introduced to better illustrate the influence of
heterogeneity on dose response.

\definecolor{myorange}{RGB}{253,184,99}
\definecolor{myorchid}{RGB}{197,179,230}
\definecolor{myskyblue}{RGB}{190,220,240}
\definecolor{myred}{RGB}{255,130,130}
\definecolor{mygray}{RGB}{220,220,220}
\definecolor{mydarkorange}{RGB}{240,180,90}

\begin{figure}[h!]
  \centering
  \begin{tikzpicture}[xscale=2]
    
    % Tissue blocks with partial transparency
    \draw[fill=myorange, draw=black, opacity=0.4]       (0,0)   rectangle (0.6,4);  % Eyelid
    \draw[fill=myorchid, draw=black, opacity=0.4]       (0.6,0) rectangle (0.9,4);  % Orbital bone
    \draw[fill=myskyblue, draw=black, opacity=0.4]      (0.9,0) rectangle (2.7,4);  % Orbital fat
    \draw[fill=myred, draw=black, opacity=0.4]          (2.7,0) rectangle (3.7,4);  % Tumour
    \draw[fill=mygray, draw=black, opacity=0.4]         (3.7,0) rectangle (4.1,4);  % Post-tumour
    \draw[fill=myorchid, draw=black, opacity=0.4]       (4.1,0) rectangle (4.4,4);  % Skull
    \draw[fill=mydarkorange, draw=black, opacity=0.4]   (4.4,0) rectangle (5.0,4);  % Deep tissue
    
    % Labels
    \node[rotate=90] at (.3,2)     {{\tiny Eyelid $\rho = 1.04$ g/cm$^3$}};
    \node[rotate=90] at (.75,2)    {{\tiny Orbital bone $\rho = 1.85$ g/cm$^3$}};
    \node[rotate=90] at (1.8,2)    {{\tiny Orbital fat $\rho = 0.3$ g/cm$^3$}};
    \node[rotate=90] at (3.2,2)    {{\tiny Tumour $\rho = 1.0$ g/cm$^3$}};
    \node[rotate=90] at (3.9,2)    {{\tiny Post-tumour $\rho = 1.0$ g/cm$^3$}};
    \node[rotate=90] at (4.25,2)   {{\tiny Skull $\rho = 1.85$ g/cm$^3$}};
    \node[rotate=90] at (4.7,2)    {{\tiny Deep tissue $\rho = 1.04$ g/cm$^3$}};
    
    % Axis
    \draw (0,0) -- (5,0);
    \node[below] at (2.5,0) {\small Depth (cm)};
    \foreach \x in {0,1,...,5}
    \draw (\x,0) -- (\x,-0.2) node[below] {\x};
  \end{tikzpicture}
  \caption{Idealised tissue composition along an ocular proton beam
    path. Layers represent major anatomical structures in the orbital
    region, with material densities indicated. The tumour lies at a
    depth of approximately 3cm. Bragg--Kleeman parameters for each
    region are given in Table~\ref{tab:range_energy}.}
  \label{fig:tissue_depth}
\end{figure}
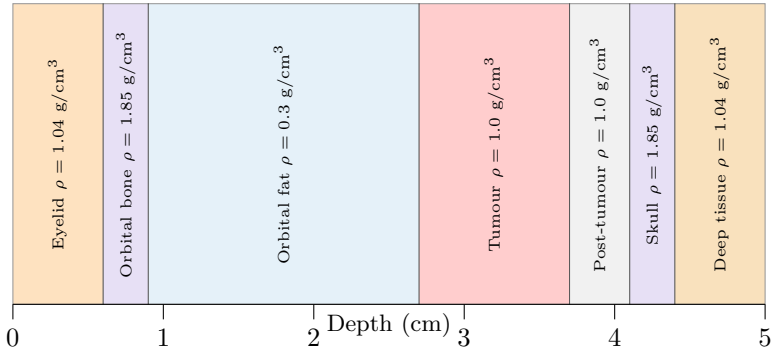

A proton beam enters through the eyelid and traverses alternating
high- and low-density regions before reaching the tumour located at
depth $z \approx 3\,\mathrm{cm}$. The heterogeneous structure
introduces strong spatial variability in stopping power and scattering
behaviour, posing challenges for standard transport solvers.

Figure~\ref{fig:ex5vi} presents the numerical results obtained using
the positivity-preserving variational inequality
scheme~\eqref{eq:discrete_inequality}. The top-left panel shows the
prescribed inflow spectrum. The middle panel displays the computed
fluence in physical space and energy, while the bottom panel reports
the resulting absorbed dose.

Despite the discontinuous material parameters, the proposed method
maintains non-negativity, avoids spurious oscillations, and captures
the correct spatial and energy localisation of the Bragg peak. The
dose rapidly attenuates beyond the tumour boundary, and the fluence
remains stable across interfaces without the need for post-processing
or filtering. These results highlight the robustness of the scheme in
resolving transport through heterogeneous media with sharp
transitions.

\begin{figure}[h!]
  \centering
  \includegraphics[width=0.7\linewidth]{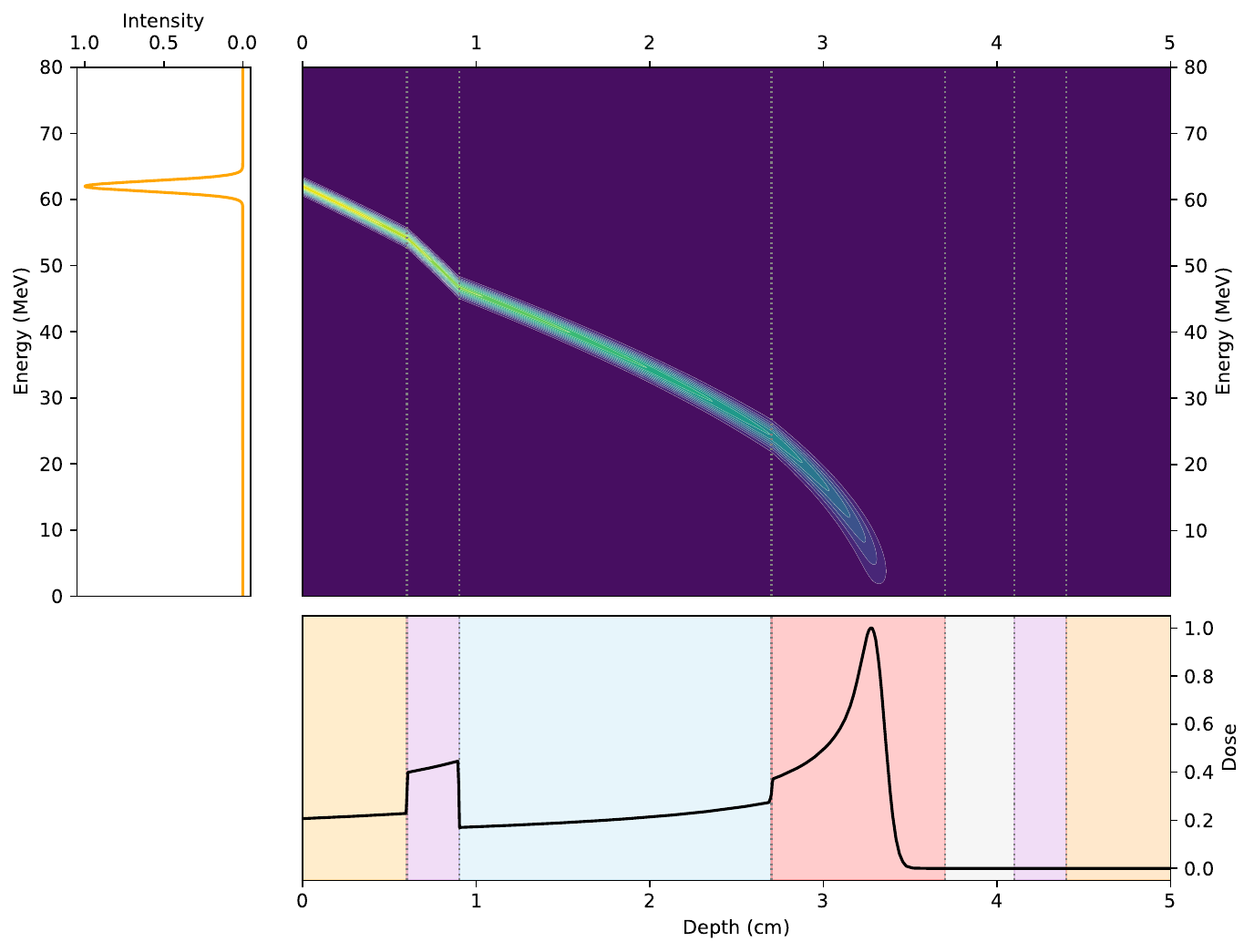}
  \caption{
    \label{fig:ex5vi}
    Simulation of a proton beam through layered heterogeneous media using the positivity-preserving scheme~\eqref{eq:discrete_inequality}. The top-left panel shows the inflow energy boundary condition. The middle panel presents the computed space-energy fluence, and the bottom panel shows the absorbed dose. The method captures material discontinuities without oscillation and preserves physical bounds across interfaces.
  }
\end{figure}

\section{Conclusion}
\label{sec:conclusion}

We have presented a deterministic finite element framework for
modelling proton transport that accounts for both inelastic energy
loss and angular scattering. A key contribution of this work is the
development of a positivity-preserving variational inequality
formulation, which ensures non-negative fluence and dose even on
coarse meshes. This formulation is compatible with streamline-upwind
Petrov--Galerkin (SUPG) stabilisation and enables accurate computation
of physically meaningful observables such as absorbed dose.

The proposed method was validated against analytic benchmarks, with
numerical experiments demonstrating optimal convergence rates,
stability under adaptive mesh refinement and robustness in the
presence of sharp material interfaces. In particular, we showed that
the variational inequality scheme eliminates spurious oscillations
that commonly arise in standard discretisations and is capable of
resolving the Bragg peak and steep energy gradients with high
fidelity. Extensions to angular diffusion models and heterogeneous
geometries confirm the versatility of the method in clinically
relevant scenarios.

This work provides a foundation for future developments in treatment
planning and robustness analysis. Immediate extensions include
dose-weighted quantities such as LET, biological models based on
survival fraction and sensitivity analysis under anatomical
uncertainty. The structure-preserving properties of the scheme also
suggest its potential compatibility with inverse planning and
optimisation frameworks.

\section*{Acknowledgements}

This work was initiated at a workshop hosted by Mathrad, supported by
the EPSRC programme grant EP/W026899/1, which subsequently funded
TP. BA and TP also received support from the Leverhulme Trust grant
RPG-2021-238 and TP the EPSRC grant EP/X030067/1. The research was
conducted by a working group sponsored by the radioprotection theme of
the Institute for Mathematical Innovation partially funding both BA and AH.

\printbibliography

\end{document}